\documentclass[final,times,authoryear,3p]{elsarticle}

\usepackage{amsmath,amsfonts,amssymb}
\usepackage{amsthm}
\usepackage{graphicx}
\usepackage{multirow}
\usepackage{color}
\usepackage{framed}
\usepackage{algpseudocode}

\usepackage{fancyhdr}
\newtheorem{theorem}{Theorem}[section]
\newtheorem{lemma}[theorem]{Lemma}
\newtheorem{proposition}[theorem]{Proposition}

\newtheorem{definition}[theorem]{Definition}
\newtheorem{remark}[theorem]{Remark}

\usepackage{amssymb}

\journal{Transportation Research Part B}

\begin{document}

\begin{frontmatter}



 \begin{center}
\textcolor{blue}{ARTICLE LINK:  http://www.sciencedirect.com/science/article/pii/S0191261515001551
\\  PLEASE CITE THIS ARTICLE AS\\ 
Han, K., Friesz, T.L., Szeto, W.Y., Liu, H., 2015. Elastic demand dynamic network user equilibrium: Formulation, existence and computation. Transportation Research Part B, 81, 183-209.}
 \line(1,0){469}
 \end{center}

\title{Elastic demand dynamic network user equilibrium: Formulation, existence and computation}


\author[ic]{Ke Han}
\ead{k.han@imperial.ac.uk}

\author[ie]{Terry L. Friesz  \corref{cor}} 
\ead{tfriesz@psu.edu}

\author[hku]{W. Y. Szeto}
\ead{ceszeto@hku.hk}

\author[ie]{Hongcheng Liu}
\ead{hql5143@gmail.com}

\cortext[cor]{Corresponding author}

\address[ic]{Department of Civil and Environmental Engineering, Imperial College London, United Kingdom.}
\address[ie]{Department of Industrial and Manufacturing Engineering, Pennsylvania State University, USA.}
\address[hku]{Department of Civil Engineering, the University of Hong Kong, China.}

\begin{abstract}
This paper is concerned with dynamic user equilibrium with elastic travel demand (E-DUE) when the trip demand matrix is determined endogenously. We present an infinite-dimensional variational inequality (VI) formulation that is equivalent to the conditions defining a continuous-time E-DUE problem. An existence result for this VI is established by applying a fixed-point existence theorem \citep{Browder} in an extended Hilbert space. We present three algorithms based on the aforementioned VI and its re-expression as a {\it differential variational inequality} (DVI): a projection method, a self-adaptive projection method, and a proximal point method. Rigorous convergence results are provided for these methods, which rely on increasingly relaxed notions of generalized monotonicity, namely {\it mixed strongly-weakly monotonicity} for the projection method; {\it pseudomonotonicity} for the self-adaptive projection method, and {\it quasimonotonicity} for the proximal point method. These three algorithms are tested and their solution quality, convergence, and computational efficiency compared. Our convergence results, which transcend the transportation applications studied here, apply to a broad family of infinite-dimensional VIs and DVIs, and are the weakest reported to date. 
\end{abstract}


\end{frontmatter}

\section{\label{Intro}Introductory remarks}

This paper is concerned with an extension of the \textit{simultaneous
route-and-departure-time} (SRDT) dynamic user equilibrium (DUE) originally
proposed in \cite{Friesz1993} and discussed subsequently by \cite{FBST,
FHNMY, FKKR, FM2014}; and \cite{FM2006}. Specifically, the model of interest
herein relaxes the assumption of fixed trip volumes by considering elastic
travel demands among origin-destination pairs. The extension of DUE based on
a fixed trip table to the explicit consideration of elastic demand is not a
straightforward matter. In particular, one has to show a dual variable
associated with arrival time is equivalent to an adjoint (co-state) variable
by exploiting the transversality conditions familiar from differential
variational inequality (DVI) theory. The first such analysis was performed
by \cite{FM2014}, who employed separable demand functions for
each origin-destination pair. They used a variational calculus approach,
which, although correct, masks many of the measure-theoretic arguments
essential to understanding the generality of a DVI representation of the
elastic-demand DUE (E-DUE) problem in continuous time. By contrast, this paper not
only considers nonseparable demand functions, but it also provides all
measure-theoretic arguments needed to understand the DVI formulation.
Furthermore, this paper presents an existence theory, three algorithms and their proofs of convergence, and numerical studies, all of which are missing from \cite{FM2014}. 
That is to say, this paper provides the first complete
mathematical and numerical analysis of the SRDT E-DUE problem.

\subsection{Dynamic user equilibrium with elastic demand: Some review}

Most of the studies of DUE reported in the \textit{dynamic traffic assignment
} (DTA) literature are about dynamic user equilibrium with constant travel
demand for each origin-destination pair. It is, of course, not generally
true that travel demand is fixed, even for short time horizons. \cite{ADL}
and \cite{YH} directly consider elastic travel demand in the context of a
single bottleneck. \cite{YM} extend a simple bottleneck model to a general
queuing network with known elastic demand functions for each
origin-destination (OD) pair. They employ a \textit{space-time expanded
network} (STEN) representation of the network loading submodel. \cite{WTC}
study a version of the dynamic user equilibrium with elastic demand, using a
complementarity formulation that requires path delays to be expressible in
closed form. \cite{SL} study dynamic user equilibrium with elastic travel
demand when network loading is based on the \textit{cell transmission model}
(CTM); their formulation is discrete-time in nature and is expressed as a
finite-dimensional variational inequality (VI). The VI is solved with a
descent method under the assumption that the delay operator is co-coercive. \cite{HUD} study dynamic user
equilibrium with elastic travel demand for a network with a single
origin-destination pair whose traffic flow dynamics are also described by
CTM; the CTM is chosen to accommodate the discrete-time complementarity
formulation of the user equilibrium model.

Although \cite{FKKR} show that analysis and computation of dynamic user
equilibrium with constant travel demand is tremendously simplified by
stating it as a differential variational inequality (DVI), they do not
discuss how elastic demand may be accommodated within a DVI framework. \cite
{FM2014} later extend the DVI formulation to an elastic demand setting,
although that paper does not discuss the existence and the computation of
E-DUE, which are our main focus here. Such a DVI formulation for the E-DUE
problem is not a straightforward extension. In particular, the DVI presented
therein has both infinite-dimensional and finite-dimensional terms.
Moreover, for any given origin-destination pair, inverse travel demand
corresponding to a dynamic user equilibrium depends on the terminal value of
a state variable representing cumulative departures. The DVI formulation
achieved in that paper is significant because it allows the still emerging
theory of differential variational inequalities to be employed for the
analysis and computation of solutions of the elastic-demand DUE problem when
simultaneous departure time and route choices are within the purview of
users, all of which constitutes a foundation problem within the field of
dynamic traffic assignment.

A good review of recent insights into abstract differential variational
inequality theory, including computational methods for solving such
problems, is provided by \cite{PS}. Also, differential variational
inequalities involving the kind of explicit, agent-specific control
variables employed herein are presented in \cite{DODG}.

\subsection{Discussion of contributions made in this paper}

In this paper, we present a unified theory and a general framework for
formulating, analyzing, and computing the simultaneous route-and-departure-time (SRDT) dynamic user equilibrium with elastic demand (E-DUE). Such an analytic framework is meant to allow qualitative analyses on E-DUE to be conducted in a rigorous manner, and to accommodate any dynamic network loading model expressible by an embedded effective delay operator.

We show, using measure-theoretic argument, that a general SRDT E-DUE can be
cast as an infinite-dimensional variational inequality problem. Unlike existing VI formulations in the literature, this VI is defined on an extended Hilbert space, which facilitates the analysis regarding existence and computation.  As a result, the existence of E-DUE is formally established in the most general setting; that is, it incorporates both route and departure time choices of travelers, and does so without invoking the  \textit{a priori} boundedness of path departure rates \footnote{We refer the reader to \cite{existence} for an illustration of the subtlety of the {\it  a priori} bound on the path departure rates when it comes to the SRDT notion of DUEs.}. 

This paper also makes a significant contribution to the computation of SRDT
E-DUE, by proposing three different algorithms and analyzing their convergence conditions. These are achieved through the VI formulation of the E-DUE model. Regarding algorithms and computation, our paper's intent is to: (i) document how far the available mathematics can take us in assuring convergence, and (ii) illustrate what can be done computationally when proceeding heuristically by relaxing monotonicity assumptions needed to assure convergence. 

In the following sections we will discuss the existence and computation of E-DUE problems in detail, while referring to the work presented in this paper and by  other scholars.

\subsubsection{Existence of SRDT E-DUE}
As commented by \cite{existence}, the most obvious approach to establishing
existence is to convert the problem to an equivalent variational inequality
problem or a fixed-point problem and then apply a version of Brouwer's famous fixed-point existence theorem. Nearly all proofs of DUE existence employ such an existence theorem, either implicitly or explicitly. One statement of Brouwer's theorem appears
as Theorem 2 of \cite{Browder}. Approaches based on Brouwer's theorem
require the set of feasible path departure rates (path flows) to be compact and
convex in a topological vector space, and typically involve the \textit{a
priori} bound on path departure rates. For instance, using the \textit{link delay model
} introduced by \cite{Friesz1993}, \cite{ZM} show that a \textit{route choice
} (RC) dynamic user equilibrium exists under certain regularity conditions.
In their modeling framework, the departure rate at each origin is given as 
\textit{a priori} and assumed to be bounded from above. Thus one is assured
that all path departure rates are automatically uniformly bounded. \cite{Mounce2007} uses Schauder's theorem to prove existence for the bottleneck model. \cite{SW1995} prove existence of a route-choice equilibrium when the path delay operator is continuous. In \cite{WTC}, the existence of a link-based dynamic user equilibrium is established under the assumption that the path departure rates are \textit{a priori} bounded.

Difficulties arise in the proof of a general existence theorem from two
aspects: (i) in a continuous-time setting, the set of feasible path departure rates is
often not compact; and (ii) the assumption of \textit{a priori} boundedness
of path departure rates, which is usually required by a topological argument, does not
arise from any behavioral argument or theory. The existence proof provided
by this paper manages to overcome these two major difficulties. Regarding
item (i) above, we employ successive finite-dimensional approximations of
the set consisting of feasible path departure rates, which allows Brouwer's fixed-point theorem to be applied. Regarding item (ii), we propose an in-depth analysis and
computation involving the path departure rates under minor assumptions on the travelers' disutility functions.

Existence result for the elastic demand case is further complicated by the
fact that the corresponding VI formulation has both infinite-dimensional and
finite-dimensional terms (see Theorem \ref{vielasticthm} below). In order to apply
Browder's theorem \citep{Browder}, one needs to work in an extended Hilbert
space that is a product of an infinite-dimensional space and a
finite-dimensional space, and define appropriate inner product that allows
compactness and weak topology to be properly defined. It is significant that
our existence result for E-DUE, stated and proven in Theorem \ref
{existencethm}, does not rely on the \textit{a priori} upper bound of path
flows and can be established for any dynamic network loading sub-model with
reasonable and weak regularity conditions.

\subsubsection{Computation of SRDT E-DUE solutions}\label{secintrocomp}

The computation of the DUE problem and its elastic-demand extension is facilitated by their equivalent mathematical formulations, such as VI, complementarity problems, fixed-point problems, and mathematical programming problems, through existing and emerging computational algorithms associated therein. The convergence of an algorithm is highly related to the property of the {\it path delay operator}, which is obtained through the dynamic network loading subproblem. In particular, most convergence proofs rest on certain types of continuity and generalized monotonicity.

 For example, \cite{JRC} develop a projection-based method to solve a {\it route-choice} DUE problem, which is a special case of the SRDT DUE problem; the convergence of this method requires continuity and strict monotonicity of the path delay operator. Similarly, a fixed-point method developed by \cite{FKKR} for SRDT DUEs relies on Lipschitz continuity and strong monotonicity. Strong monotonicity is known to not hold for general networks and DNL models \citep{MS2007}, and algorithms that rely on relaxed notions of monotonicity have been proposed in the literature. \cite{LS2002a} and \cite{SL} develop an alternating direction method and a descent method for cell-based DUE problems with fixed and elastic demand, respectively; both methods require that the delay operator is co-coercive. According to \cite{ZH}, a sufficient condition for co-coerciveness includes Lipschitz continuity and monotonicity (rather than strong monotonicity). Algorithms with even more relaxed convergence requirements, such as the day-to-day route swapping algorithm \citep{HL, SL2006, THG} and the extragradient method \citep{LHGS}, are also proposed for solving DUE problems. In particular, convergence of the route swapping algorithm requires continuity and monotonicity \citep{MC2011}; and the extragradient method requires Lipschitz continuity and pseudo monotonicity for convergence.

This paper provides three new algorithms for computing E-DUE problems with convergence proofs that extend the knowledge of mathematical tools available for ensuring convergence. The first algorithm is a {\it projection} algorithm, which is based on  the new VI formulation proposed in this paper. It relies on a minimum-norm projection operator onto a subset of an extended Hilbert space. Notably, this new projection operator may be explicitly represented by invoking the mathematical paradigm of differential variational inequality (DVI) and linear-quadratic optimal control, which is original in the literature. Regarding convergence of the projection method, instead of relying on the well-known strong monotonicity, we propose a new condition called {\it mixed strongly-weakly monotonicity} (MSWM). The MSWM stipulates strong monotonicity along a subset of paths, and weak monotonicity along the rest of the paths; it is one type of component-wise monotonicity, and has been extensively studied in the literature of dynamical systems and optimization. This paper provides the first convergence result for the MSWM property. The second algorithm is a {\it self-adaptive projection method}, which is originally proposed by \cite{HL2002} for solving generic variational inequalities. This method converges given that the delay operator is continuous and pseudo monotone, which is a weaker notion than (strong) monotonicity. Similar to the first computational method, this self-adaptive projection method utilizes the aforementioned explicit instantiation of the projection onto the extended Hilbert space. Finally, the third algorithm is called {\it proximal point method} \citep{Konnov}. It replaces the original VI problem, which may not satisfy the require monotonicity, with a sequence of regularized problems, each of which may be solved with standard algorithms due to improved regularity (monotonicity).  In this paper we provide a new convergence proof for the proximal point method for a class of {\it quasi monotone} delay operators, and further relax the monotonicity condition compared to the existing literature. The three new algorithms are tested on two networks, which numerically demonstrates their convergence, solution quality, and computational complexity.

The key contributions made in this paper include:

\begin{enumerate}
\item[-] the expression of the simultaneous route-and-departure-time (SRDT)
E-DUE problem as an infinite-dimensional variational inequality in an extended Hilbert space;

\item[-] a general existence result for the E-DUE when both departure time and route choices are within the purview of travelers, which does not invoke the {\it a priori} boundedness on the path departure rates;

\item[-] three new algorithms for computing E-DUE problems with rigorous convergence proofs that rely on increasingly relaxed monotonicity conditions on the delay operator.

\end{enumerate}

The rest of this paper is organized as follows. Section \ref{secNEB} offers an introduction to some key concepts, notations, and mathematical background needed to present and analyze the E-DUE problem. Section \ref{seceduedefvi} presents the variational inequality formulation of the E-DUE problem in an extended Hilbert space. We provide an existence theory for the E-DUE problem in Section \ref{secedueexistence}. Section \ref{seccomputation} presents the three computational methods and their convergence results. Section \ref{secnum} presents a numerical study of the proposed computational methods. Finally, some concluding remarks are offered in Section \ref{secconclusion}.

\section{Notation and essential background}\label{secNEB}

Throughout his paper, the time interval of analysis is a single commuting period expressed as $[t_{0},\, t_{f}] \subset \mathbb{R}$
where $t_{f}>t_{0}$, and both $t_{0}$ and $t_{f}$ are fixed. We let $\mathcal{P}$ be the set of all paths utilized by travelers. For each $p\in\mathcal{P}$, we define the path departure rate (in vehicle per unit time), which is a function of departure time $t\in[t_0,\,t_f]$, $h_p(\cdot):~[t_0,\,t_f] \to\mathbb{R}_+$, where $\mathbb{R}_+$ denotes the set of non-negative real numbers. Each path departure rate $h_p(t)$ is interpreted as a path flow measured at the entrance of the first link of the relevant path. We next define $h(\cdot)=\{h_p(\cdot)\,:p\in\mathcal{P}\}$ to be a vector of departure rates, which is viewed as a vector-valued function of $t$, the departure time \footnote{For notation convenience and without causing any confusion, we will sometimes use $h$ instead of $h(\cdot)$ to denote the path departure rate vectors.}. 

We let $L^2[t_0,\,t_f]$ be the space of square-integrable functions define on the interval $[t_0,\,t_f]$, and $L^2_+[t_0,\,t_f]$ be its subset consisting of non-negative functions. It is stipulated that each path departure rate is square integrable: $h_p(\cdot)\in L^2_+[t_0,\,t_f]$ and $h(\cdot)\in (L^2_+[t_0,\,t_f])^{|\mathcal{P}|}$, where  $\big(L^2[t_0,\,t_f]\big)^{|\mathcal{P}|}$ is the $|\mathcal{P}|$-fold product of the Hilbert space $L^2[t_0,\,t_f]$, and $\big(L_+^2[t_0,\,t_f]\big)^{|\mathcal{P}|}$ is its subset consisting of non-negative path departure vectors. The inner product on the Hilbert space $\big(L^2[t_0,\,t_f]\big)^{|\mathcal{P}|}$ is defined as
\begin{equation}\label{ipdef}
\left<h^1,\, h^2\right>~=~\int_{t_0}^{t_f}\left(h^1(t)\right)^T\,h^2(t)\,dt~=~\sum_{p\in\mathcal{P}}\int_{t_0}^{t_f} h^1_p(t) \cdot h^2_p(t)\,dt
\end{equation}
where the superscript $T$ denotes the transpose of vectors. Moreover, the norm 
\begin{equation}\label{l2normdef}
\left\|u\right\|_{L^2}~=~\left<u,\,u\right>^{1/2}
\end{equation}
\noindent is induced by the inner product \eqref{ipdef}.

Here, as in all DUE modeling, the single most crucial ingredient is the path delay operator, which maps a given vector of departure rates $h$ to a vector of path travel times. More precisely, we let
\begin{equation*}
D_{p}(t,\,h)\qquad \forall t\in[t_0,\,t_f],\quad  \forall p\in \mathcal{P}
\end{equation*}
be the path travel time of a driver departing at time $t$ and following path $p$, given the departure rates associated with all the paths in the network, which is expressed by $h$ in the expression above. We then define the path delay operator $D(\cdot)$ by letting $D(h)=\{D_p(\cdot,\,h):\,p\in\mathcal{P}\}$, which is a vector of time-dependent path travel times $D_p(t,\,h)$.  $D(\cdot)$ is an operator defined on $\big(L_+^2[t_0,\,t_f]\big)^{|\mathcal{P}|}$, and maps a vector valued function $h(\cdot)$ to another vector-valued function $\{D_p(\cdot,\,h):\,p\in\mathcal{P}\}$. In summary,
\begin{equation}\label{Ddef}
D:~\big(L^2_+[t_0,\,t_f]\big)^{|\mathcal{P}|}~\rightarrow~\big(L^2_+[t_0,\,t_f]\big)^{|\mathcal{P}|},\qquad h(\cdot)=\{h_p(\cdot),~p\in\mathcal{P}\}~\mapsto~D(h)=\{D_p(\cdot,\,h),~p\in\mathcal{P}\}
\end{equation}
\noindent The {\it effective} path delay operator $\Psi$ is similarly defined, except that the effective path delay contains, in addition to path travel time, also arrival penalties. As such, the effective path delay is a more general notion of ``travel cost" than path delay. The effective delay operator is defined as follows. 
\begin{equation}\label{rvPsidef}
\Psi:~\big(L^2_+[t_0,\,t_f]\big)^{|\mathcal{P}|}~\rightarrow~\big(L^2_+[t_0,\,t_f]\big)^{|\mathcal{P}|},\qquad h(\cdot)=\{h_p(\cdot),~p\in\mathcal{P}\}~\mapsto~\Psi(h)=\{\Psi_p(\cdot,\,h),~p\in\mathcal{P}\}
\end{equation}
\noindent where
\begin{equation}\label{cost}
\Psi _{p}(t,\,h)~=~D_{p}(t,\,h)+f\big( t+D_{p}(t,h)-T_{A}\big) \qquad \forall t\in[t_0,\,t_f],\quad \forall p\in \mathcal{P}
\end{equation}
where $T_{A}$ is the desired arrival time and $T_{A}<t_{f}$. The term $f\big( t+D_{p}(t,h)-T_{A}\big)$ assesses a nonnegative penalty whenever
\begin{equation}
t+D_{p}(t,\,h)~\neq~T_{A}  \label{pgt}
\end{equation}
since $t+D_{p}(t,\,h)$ is the clock time at which departing traffic arrives at
the destination of path $p\in \mathcal{P}$. Note that, for convenience here, $T_A$ is assumed to be independent of path or origin-destination pair. However, that assumption is easy to relax, and the consequent generalization of our model is a trivial extension.

We interpret $\Psi_p(t,\,h)$ as the perceived travel cost of drivers departing at time $t$ following path $p$ given the vector of path departure rates $h$.  We stipulate that for all $p\in\mathcal{P}$, the function $\Psi _{p}(\cdot,\,h)$ is measurable, almost everywhere positive, and square integrable. The notation
\begin{equation*}
\Psi(h)~\doteq~\{\Psi _{p}(\cdot ,h):~p\in \mathcal{P}\} \in \big(L_+^2[t_0,\,t_f]\big)^{|\mathcal{P}|}
\end{equation*}
is used to express the complete vector of effective path delays.

 In order to develop an appropriate notion of minimum travel cost in the measure-theoretic context, we require the concept of {\it essential infimum}. In particular,  for  any measurable function $g: ~[t_0,\,t_f]\rightarrow \mathbb{R}$, the essential infimum of $g(\cdot)$ on $[t_0,\,t_f]$ is given by
\begin{equation}\label{chapVI:essinf}
\hbox{essinf}\left\{g(s):~s\in [t_0,\,t_f]\right\}~\doteq~\sup\left\{ x\in\mathbb{R}:~meas\{s\in [t_0,\,t_f]:~g(s)<x\}~=~0\right\}
\end{equation}
\noindent where $meas$ represents the (Lebesgue) measure. Note that for each $x>\hbox{essinf}\{g(s):~s\in [t_0,\,t_f]\}$ it must be true by definition that 
$$
meas\{s\in [t_0,\,t_f]:~f(s)<x\}>0
$$

\noindent Let us  define the essential infimum of the effective path delay, which depend on the path departure rate vector $h$,
\begin{equation}\label{essinfdef1}
v_p(h)~=~\hbox{essinf}\left\{\Psi_p(t,\,h):~t\in[t_0,\,t_f]\right\}~>~0\qquad\forall p\in\mathcal{P}
\end{equation}
\noindent The minimum travel cost within a given O-D pair $(i,\,j)$ is thus defined as
\begin{equation}\label{essinfdef2}
v_{ij}(h)~=~\min\left\{v_p(h): ~p\in \mathcal{P}_{ij}\right\}~>~0 \qquad \forall \left( i,j\right) \in \mathcal{W}
\end{equation}
\noindent By definition, $v_{ij}(h)$ is the minimum travel cost within O-D pair $(i,\,j)$ for all route choices and departure time choices.

The (effective) path delay operator is a key component of analytical dynamic user equilibrium (DUE) models; it is usually not available in closed form and has to be numerically evaluated from dynamic network loading (DNL), which is a sub-problem of a complete DUE model. The DNL sub-problem aims at describing and predicting the spatial-temporal evolution of traffic flows on a network that is consistent with established route and departure time choices of travelers, by introducing appropriate dynamics to flow propagation, flow conservation, and travel delays on a network level.  Any DNL must be consistent with the established path departure rates and link flow model; and it is usually performed under the {\it first-in-first-out} (FIFO) rule. A few link flow models commonly employed for the DNL procedure include the link delay model \citep{Friesz1993}, the Vickrey model \citep{Vickrey, GVM1, GVM2}, the cell transmission model \citep{CTM1, CTM2}, the link transmission model \citep{LTM, LKWM}, and the Lighthill-Whitham-Richards model \citep{LW, Richards}.

Work regarding the dynamic network loading models dates back to the 1990s with a significant number of publications \citep{FKKR, FHLY, FHNMY, Handissertation,   LS, NZ2010, Szeto2003, Szeto2011, SL, SL2006} and \cite{Ukkusuri}. Notably, despite the absence of closed-form representations of the delay operators, it has been reported that certain dynamic network loading models can be expressed as a system of {\it differential algebraic equations} (DAEs) or {\it partial differential algebraic equations} (PDAEs). Those results include: the DAE system formulation of the DNL procedure for the link delay model \citep{FKKR}; the DAE system formulation of the DNL procedure for the Vickrey model \citep{Handissertation}; the DAE system formulation of the DNL procedure for the LWR-Lax model \citep{FHNMY}; and the PDAE/DAE system formulation of the general LWR model \citep{spillback}.

\section{Definition of SRDT E-DUE and the variational inequality formulation}\label{seceduedefvi}

\subsection{Dynamic user equilibrium with elastic demand}

We introduce the trip matrix $\big(Q_{ij}: (i,\,j)\in\mathcal{W}\big)$, where each $Q_{ij}\in \mathbb{R} _{+}$ is the (elastic) travel demand between the origin-destination (O-D) pair $\left( i,j\right) \in \mathcal{
W}$, and $\mathcal{W}$ is the set of origin-destination pairs. Note that unlike route-choice DUE models, $Q_{ij}$ here represents total traffic volume, instead of flow that changes over time.   The flow conservation constraints read
\begin{equation}\label{cons}
\sum_{p\in \mathcal{P}_{ij}}\int_{t_0}^{t_f}h_p(t)\,dt~=~Q_{ij}\qquad\forall (i,\,j)\in\mathcal{W}
\end{equation}
\noindent where \eqref{cons} consists of Lebesgue integrals, and $\mathcal{P}_{ij}\subset \mathcal{P}$ is the set of paths connecting O-D pair $(i,\,j) \in \mathcal{W}$.  In the elastic demand case, the travel demand between O-D pair $(i,\,j)$ is assumed to be expressed as the following invertible function
$$
Q_{ij}~=~G_{ij}[v]\qquad\forall (i,\,j)\in\mathcal{W}
$$
\noindent where  $v=\{v_{ij}(h):~(i,\,j)\in\mathcal{W}\}$ is the vector of O-D specific minimum travel costs $v_{ij}(h)$ defined in \eqref{essinfdef1}-\eqref{essinfdef2}. We note that $Q_{ij} $ is the \underline{unknown} travel demand
between $(i,\,j)$ that must ultimately be achieved by the end of the time horizon $t=t_{f}$. We will find it convenient to form the complete vector of travel
demands by concatenating the O-D specific travel demands to obtain
$$
Q ~=~\big( Q_{ij} :~( i,\, j) \in \mathcal{W}\big)~=~\big(G_{ij}[v] :~( i,\, j) \in \mathcal{W}\big) \in \mathbb{R}_+^{\left\vert \mathcal{W}\right\vert } 
$$
which defines a mapping from $v$ to $Q$ that, when invertible, gives rise to the {\it inverse demand function}:
$$
\Theta: \mathbb{R}_+^{|\mathcal{W}|}~\to~  \mathbb{R}_{++}^{|\mathcal{W}|},\qquad Q~\mapsto~v
$$
\noindent where
\begin{equation}\label{inversedemandfcn}
v~=~\big(v_{ij}:~(i,\,j)\in\mathcal{W}\big)\quad\hbox{and}\quad v_{ij}~=~\Theta_{ij}[Q]
\end{equation}
\noindent Notice that the inverse demand function defined in \eqref{inversedemandfcn} is non-separable in the sense that each minimum O-D travel cost $v_{ij}$ is jointly determined by the entire vector of elastic demands $Q=\big(Q_{ij}:\,(i,\,j)\in\mathcal{W}\big)$.

 Accordingly, we employ the following feasible set of departure flows when the travel demand between each origin-destination pair is endogenous. 
\begin{equation}\label{tildelambdadef}
\widetilde\Lambda~=~\left\{(h,\,Q):~h\geq 0,~~ \sum_{p\in\mathcal{P}_{ij}}\int_{t_0}^{t_f} h_p(t)\,dt~=~Q_{ij} \quad \forall (i,\,j)\in\mathcal{W}\right\}
\subset\left(L^2[t_0,\,t_f]\right)^{|\mathcal{P}|}\times \mathbb{R}^{|\mathcal{W}|}
\end{equation}
\noindent where $\Big(L^2[t_0,\,t_f]\Big)^{|\mathcal{P}|}\times \mathbb{R}^{|\mathcal{W}|}$ is the direct product  of the $|\mathcal{P}|$-fold product of Hilbert spaces, and the $|\mathcal{W}|$-dimensional Euclidean space consisting of vectors of elastic travel demands. 

With preceding preparation, we are in a place where the simultaneous route-and-departure-time dynamic user equilibrium with elastic demand can be rigorously defined, as follows.

\begin{definition}\label{dueelasticdef}
{\bf (Dynamic user equilibrium with elastic demand)} A pair $(h^*,\, Q^*)\in \widetilde \Lambda$ is said to be a dynamic user equilibrium with elastic demand if for all $(i,\,j)\in\mathcal{W}$,
\begin{align}\label{chapVI:eqn1}
& h_p^*(t)~>~0,~~p\in\mathcal{P}_{ij}~~\Longrightarrow~~\Psi_p(t,\,h^*)~=~\Theta_{ij}[Q^*]\qquad\hbox{for almost every } t\in[t_0,\,t_f]  \\
\label{chapVI:eqn2}
& \Psi_p(t,\,h^*)~\geq~\Theta_{ij}[Q^*] \qquad\hbox{for almost every } t\in[t_0,\,t_f], \quad   \forall p\in\mathcal{P}_{ij}
\end{align}

\end{definition}

\subsection{The variational inequality formulation of the SRDT E-DUE problem}\label{subseceduevi}

Experience with differential games suggests that the DUE problem with elastic demand can be expressed as a variational inequality, as shown in the theorem below. 

\begin{theorem}\label{vielasticthm} {\bf (E-DUE equivalent to a variational inequality)} Assume $\Psi_p(\cdot,\,h): [t_0,\,t_f]\rightarrow \mathbb{R}_{++}$ is positive and measurable for all $p\in\mathcal{P}$ and all $h$ such that $(h,\,Q)\in\widetilde \Lambda$. Also assume that the inverse demand function $\Theta_{ij}[\cdot]$ exists for all $(i,\,j)\in \mathcal{W}$. Then a pair $(h^*,\,Q^*)\in\widetilde\Lambda$ is a DUE with elastic demand (Definition \ref{dueelasticdef}) if and only if it solves the following variational inequality:
\begin{equation}\label{chapVI:elasticvi}
\left.\begin{array}{c}
\hbox{find} ~(h^{\ast},\,Q^*)\in \widetilde\Lambda~\hbox{such that}\\
\displaystyle\sum_{p\in \mathcal{P}}\int\nolimits_{t_{0}}^{t_{f}}\Psi _{p}(t,h^*)\big(h_{p}(t)-h_{p}^{\ast}(t)\big)\,dt
-\sum_{(i,\,j) \in \mathcal{W}}\Theta _{ij}\left[ Q^{\ast } \right] (Q_{ij} -Q_{ij}^{*}) ~\geq~0   \\ \forall (h,\,Q)\in \widetilde\Lambda  
\end{array} 
\right\} VI\big(\Psi,\, \Theta,\, [t_{0},\, t_{f}]\big)
\end{equation}
\end{theorem}
\begin{proof}
The proof is postponed until \ref{secappVIthm}.
\end{proof}

In the remainder of this section, we rewrite \eqref{chapVI:elasticvi} in a more compact and generic VI form by introducing an extended Hilbert space. We consider the product space $E\doteq \big(L^2[t_0,\,t_f]\big)^{|\mathcal{P}|}\times \mathbb{R}^{|\mathcal{W}|}$, which is a space with the induced inner product defined as follows 
\begin{equation}\label{inducedipdef}
\begin{array}{c}
\left<X,\,Y\right>_E~\doteq~\displaystyle\sum_{i=1}^{|\mathcal{P}|}\int_{t_0}^{t_f}\xi_i(t)\cdot \eta_i(t)\,dt+\sum_{j=1}^{|\mathcal{W}|}u_jv_j\\
\displaystyle X~=~\left(\xi_1(\cdot),\,\ldots,\,\xi_{|\mathcal{P}|}(\cdot),\, u_1,\,\ldots,\,u_{|\mathcal{W}|}\right)\in E\\
\displaystyle Y~=~\left(\eta_1(\cdot),\,\ldots,\,\eta_{|\mathcal{P}|}(\cdot),\, v_1,\,\ldots,\,v_{|\mathcal{W}|}\right)\in E
\end{array}
\end{equation}
The inner product $\left<\cdot,\,\cdot\right>_E$ immediately leads to the following norm on the space $E$
\begin{equation}
\|X\|_E~=~\left<X~,~X\right>_E^{1/2}
\end{equation}
making $E$ a metric space. In the following proposition, we show that the inner product $\left<\cdot,\,\cdot\right>_E$ and the norm $\|\cdot\|_E$ are well defined, and the resulting extended space $E$ is indeed a Hilbert space. 
\begin{proposition}\label{propE}
The inner product $\left<\cdot,\,\cdot\right>_E$ and norm $\|\cdot\|_E$ are well defined. In addition, the space $E$, equipped with $\left<\cdot,\,\cdot\right>_E$ and the induced metric, is a Hilbert space.
\end{proposition}
\begin{proof}
A well-defined inner product must satisfy, for all $X,\,Y,\,Z\in E$, the following properties:
\begin{enumerate}
\item symmetry, i.e. $\left<X,\,Y\right>_E=\left<Y,\,X\right>_E$;
\item linearity, i.e. $\left<aX,\,Y\right>_E=a\left<X,\,Y\right>_E$ for all $a\in\mathbb{R}$, and $\left<X+Y,\,Z\right>_E=\left<X,\,Z\right>+\left<Y,\,Z\right>_E$; 
\item positive-definiteness, i.e. $\left<X,\,X\right>_E\geq 0$ and $\left<X,\,X\right>_E=0 \Rightarrow X=0$.
\end{enumerate}
\noindent It is straightforward to verify that the inner product defined in \eqref{inducedipdef} satisfies all these three conditions, and thus is well defined. Consequently, the induced norm $\|\cdot\|_E$ is also well defined. Finally, since both $\big(L^2[t_0,\,t_f]\big)^{|\mathcal{P}|}$ and $\mathbb{R}^{|\mathcal{W}|}$ are complete metric spaces, their product space $E$ is also a complete metric space, and hence a Hilbert space. 
\end{proof}

The set $\widetilde\Lambda$ of  admissible pairs $(h,\,Q)$ can now be embedded in the extended space $E$.  In view of the inverse demand function $\Theta=(\Theta_{ij}[\cdot]:\,(i,\,j)\in\mathcal{W})$, we introduce  the following notation.
$$
\Theta^-~\doteq~(-\Theta_{ij}[\cdot]:\,(i,\,j)\in\mathcal{W}):\quad\mathbb{R}_+^{|\mathcal{W}|}\longrightarrow \mathbb{R}_{--}^{|\mathcal{W}|}
$$
\noindent Consequently, we define the mapping 
\begin{equation}\label{calfdef}
\mathcal{F}: \widetilde \Lambda \longrightarrow E,\qquad (h,\,Q)~\mapsto~\left(\Psi(h)~,~ \Theta^-[Q]\right)
\end{equation}
\noindent where $(h,\,Q)\in\widetilde \Lambda$, $\Psi(h)\in\big(L_+^2[t_0,\,t_f]\big)^{|\mathcal{W}|},\, \Theta^-[Q]\in\mathbb{R}_{--}^{|\mathcal{W}|}$. Such a mapping is clearly well defined.  With the preceding discussion, the VI formulation \eqref{chapVI:elasticvi} is readily rewritten as the following variational inequality in the extended Hilbert space.

\begin{equation}\label{viextend}
\left. 
\begin{array}{c}
\text{find }X^{\ast }\in \widetilde\Lambda \text{ such that} \\ \\
\left<\mathcal{F}(X^*),\, X-X^*\right>_E~\geq~0 \\ \\
\forall X\in \widetilde\Lambda %
\end{array}%
\right\} VI\big(\mathcal{F},\, [t_0,\, t_f]\big) 
\end{equation}%
where $X=(h,\,Q)$ and $X^*=(h^*,\,Q^*)$. Problem (\ref{viextend}) is expressed in the most generic form of a variational inequality, which allows analyses regarding solution existence and computation to be carried out in a framework well supported by the VI theory in mathematical programming.

\section{Existence of SRDT E-DUE}\label{secedueexistence}

In this section we establish the existence result for \eqref{viextend}, which is an equivalent formulation of the E-DUE problem. Our proposed approach is meant to incorporate the most general dynamic network loading sub-model with minimum regularity requirements, and to yield existence of E-DUE without invoking the {\it a priori} upper bound on the path departure rates.

The analysis regarding solution existence for the variational inequality (\ref{viextend}) is based on the following theorem that extends Brouwer's fixed point existence theorem to topological vector spaces. 

\begin{theorem}\label{browderthm} 
\citep{Browder} Let $K$ be a compact convex subset of the locally convex
topological vector space $V$, $T$ a continuous (single-valued) mapping of $K$
into $V^{\ast }$, where $V^*$ is the dual space of $V$. Then there exists $u_{0}$ in $K$ such that 
\begin{equation*}
\Big<T(u_{0}),\,u-u_0\Big>~\geq ~0\qquad\forall u\in K
\end{equation*}%
\end{theorem}
\begin{proof}
See \cite{Browder}.
\end{proof}

In preparation for our existence proof, we recap several key results from functional analysis that facilitate our presentation. In particular, we note the following facts without proofs. The reader is referred to \cite{Royden} for more detailed discussion on these subjects. 

\begin{proposition}
The space $L^2[t_0,\,t_f]$ is a locally convex topological vector space. In addition, the $|\mathcal{P}|$-fold product of this space, denoted by $\big(L^2[t_0,\,t_f]\big)^{|\mathcal{P}|}$, is also a locally convex topological vector space. Finally, the product space  $\big(L^2[t_0,\,t_f]\big)^{|\mathcal{P}|}\times \mathbb{R}^{|\mathcal{W}|}$ is again a locally  convex topological vector space.
\end{proposition}

\begin{proposition}
The dual space of $\big(L^{2}[t_0,\,t_f]\big)^{|\mathcal{P}|}$  has a natural isomorphism with $\big(L^{2}[t_0,\,t_f]\big)^{|\mathcal{P}|}$. The dual space of the Euclidean space $\mathbb{R}^{|\mathcal{W}|}$ consisting of columns of $|\mathcal{W}|$ real numbers is interpreted as the space consisting of rows of $|\mathcal{W}|$ real numbers. As a consequence, the dual space of $\big(L^2[t_0,\,t_f]\big)^{|\mathcal{P}|}\times\mathbb{R}^{|\mathcal{W}|}$ is again $\big(L^2[t_0,\,t_f]\big)^{|\mathcal{P}|}\times\mathbb{R}^{|\mathcal{W}|}$.
\end{proposition}

\begin{proposition}\label{propsition3}
In a metric space (and topological vector space), the notion of compactness is equivalent to the notion of sequential compactness; that is, every infinite sequence has a convergent subsequence. 
\end{proposition}

Theorem \ref{browderthm} is immediately applicable for showing that \eqref{viextend} has a solution if (1) $\mathcal{F}$ is continuous; and (2) $\widetilde \Lambda\subset E$ is compact. The continuity of $\mathcal{F}$ amounts to the continuity of the path delay operator $\Psi$, which has been shown in the literature for different link flow dynamics and network extensions. These results include the continuity result with the {\it link delay model} \citep{Friesz1993} established in \cite{LDMcont}; the continuity result with the {\it Vickrey model} \citep{Vickrey}  established in \cite{existence}; and the continuity result with the LWR-Lax model \citep{FHNMY} established in \cite{BH1}.

 Unfortunately, the second condition involving compactness does not generally hold for the infinite-dimensional (continuous-time) problem we study here. To overcome such an obstacle, we proceed in a similar way as in \cite{existence} by considering finite-dimensional approximations of the underlying infinite-dimensional Hilbert space. Another major hurdle that stymied many researchers is the {\it a priori} upper bound on the path departure rates. Such a bound is important for a topological argument that we will rely on in the proof, but does not arise from any physical or behavioral perspective of traffic modeling. In fact, as observed by \cite{BH}, the equilibrium path flows could very well become unbounded or even contain Dirac-delta \footnote{In this case, no solution in the sense of Definition \ref{dueelasticdef} exists.}, if no additional assumptions are made regarding exogenous parameters of the differential Nash game, such as travelers' disutility functions.

The following assumptions are key to our existence result. The first assumption, {\bf (A1)}, poses reasonable hypothesis on drivers' perceived arrival costs. In particular, \eqref{fassumption} holds on the basis that the unit cost of early arrival is always less than or equal to the unit cost of elapsed travel time \citep{Small}. The reader is also referred to \cite{existence} for the motivation and generality of such an assumption.  The second assumption, {\bf (A2)}, is concerned with the DNL model and can be easily verified by  models such as the Vickrey model \citep{Vickrey, GVM1, GVM2}, the LWR-Lax model \citep{FHNMY}, the cell transmission model \citep{CTM1, CTM2}, the link transmission model \citep{LTM}, and the Lighthill-Whitham-Richards model \citep{LW, Richards}. Assumption {\bf (A3)} is one instantiation of the continuity of the effective delay operator $\Psi$, and has been shown to hold for several network loading models. These include the Vickrey model \citep{Vickrey, GVM1, GVM2} and the LWR-Lax model \citep{FHNMY}, for which {\bf (A3)} is established by \cite{BH1}; and the link delay model \citep{Friesz1993}, for which {\bf (A3)} is shown in \cite{LDMcont}. We refer the reader to these references for detailed proofs of {\bf (A3)}.

\begin{itemize}
\item[{\bf (A1).}] The function $f(\cdot )$ appearing in \eqref{cost} is
continuous on $[t_0,\,t_f]$ and satisfies 
\begin{equation}\label{fassumption}
f(t_2)-f(t_1)~\geq~\Delta (t_2-t_1) \qquad \forall~ t_0~\leq~t_1~<~t_2~\leq~t_f
\end{equation}
for some $\Delta >-1$\\

\item[ {\bf (A2).}] The {\it first-in-first-out} (FIFO) rule is obeyed on the path level. In addition, each link $a\in \mathcal{A}$ in the network has
a finite exit flow capacity $M_{a}~<~\infty $. \\

\item[{\bf (A3).}] For any sequence of departure rate vectors $\{h^n(\cdot)\}_{n\geq 1}$ that are uniformly bounded point wise by a positive constant and  converge weakly to $h^*\in \big(L^2[t_0,\,t_f]\big)^{|\mathcal{P}|}$, the corresponding effective path delays $\Psi_p(t,\,h^n)$ converge to $\Psi_p(t,\,h^*)$ uniformly for all $p\in\mathcal{P}$ and $t\in[t_0,\,t_f]$.

\end{itemize}

\noindent We are now ready to state and prove the main result of this section.

\begin{theorem}\label{existencethm}{\bf (Existence of E-DUE)}
Let assumptions {\bf (A1)}-{\bf (A3)} hold. If, in addition,  the inverse demand function $\Theta: \mathbb{R}_+^{|\mathcal{W}|}\rightarrow \mathbb{R}_{++}^{|\mathcal{W}|}$ is continuous, then the E-DUE problem has a solution.
\end{theorem}
\begin{proof}
The proof is postponed until \ref{secappexistencethm}.
\end{proof}

\section{Computation of the E-DUE problem}\label{seccomputation}

Computation of the elastic demand DUE is most facilitated by the variational inequality (VI) formulation proposed in this paper. Some methods commonly seen for solving VIs or equivalent mathematical forms include the projection method or the fixed-point method, among others; see Section \ref{secintrocomp} for a review of more methods. Another method less known to the traffic research community is the {\it proximal point method} (PPM) \citep{Konnov, Allevi}, which will be detailed in this section.

We present three computational methods for the E-DUE model, while analyzing their convergence conditions based on generalized monotonicity of the delay operator. These methods are: the projection (fixed-point) method, the self-adaptive projection method, and the proximal point method.

\subsection{The projection method}\label{subsecprojection}

The projection method for solving the VI \eqref{viextend} requires the following iterative process
\begin{equation}\label{projection}
X^{k+1}~=~P_{\widetilde\Lambda}\left[ X^{k} - \alpha \mathcal{F}(X^k)\right]
\end{equation}
where $P_{\widetilde \Lambda}[\cdot]$ is the minimum-norm projection onto the convex set $\widetilde\Lambda$. Here, $X^{k}=(h^k,\,Q^k)$ belongs to the extended Hilbert space $E$. In order to explicitly describe this projection operator in an infinite-dimensional and extended Hilbert space, we invoke the {\it differential variational inequality} (DVI) formalism \citep{FM2014}. This is done by observing that the elastic demand satisfaction constraint \eqref{tildelambdadef} can be easily rewritten as a two-point boundary value problem, leading to the following equivalent definition of the feasible set 
\begin{equation}
\widetilde\Lambda_1\doteq\left\{(h,\,Q): {d y_{ij}(t)\over dt}=\sum_{p\in\mathcal{P}_{ij}}h_p(t),~~ y_{ij}(t_0)=0,~~ y_{ij}(t_f)=Q_{ij} ~~\forall (i,\,j)\in\mathcal{W}\right\}~\subset \big(L^2[t_0,\,t_f]\big)^{|\mathcal{P}|}\times \mathbb{R}^{|\mathcal{W}|}
\end{equation}
The next theorem establishes an explicit formula for the right hand side of \eqref{projection}, based on the DVI theory.

\begin{theorem}\label{thmexplicitproj}
Given any $X^k=(h^k,\,Q^k)\in\widetilde\Lambda$, then $X^{k+1}$ appearing on the left hand side of \eqref{projection} can be expressed as $X^{k+1}=(h^{k+1},\,Q^{k+1})$ where
\begin{align}
\label{projex1}
h_p^{k+1}(t)~=~&\left[h_p^k(t)-\alpha\Psi_p(t,\,h^k)+Q_{ij}^k+\alpha\Theta_{ij}[Q^k]-Q_{ij}^{k+1}\right]_+ \qquad\forall p\in\mathcal{P}
\\
\label{projex2}
Q_{ij}^{k+1}~=~&\sum_{p\in\mathcal{P}_{ij}}\int_{t_0}^{t_f}\Big[h_p^k(t)-\alpha\Psi_p(t,\,h^k)+Q_{ij}^k+\alpha\Theta_{ij}\big[Q^k\big]-Q_{ij}^{k+1}\Big]_+ dt \qquad\forall (i,\,j)\in\mathcal{W}
\end{align}
In addition, such $h^{k+1}$ and $Q^{k+1}$ are unique.
\end{theorem}
\begin{proof}
The proof will be provided in \ref{secappthmexplicitproj}.
\end{proof}

The following pseudo code summarizes the projection algorithm.

\begin{framed}
\noindent {\bf Projection algorithm}
\begin{description} 

\item[Step 0] Identify an initial feasible point $X^0=(h^0,\,Q^0)\in\widetilde\Lambda$.  Set the iteration counter $k=0$. 

\item[Step 1] Solve the dynamic network loading problem with path departure rate vector given by $h^k$, and obtain the effective path delays $\Psi_p(t,\,h^k),~\forall t\in[t_0,\,t_f],\, p\in\mathcal{P}$. Find $Q^{k+1}$ and then $h^{k+1}$ according to \eqref{projex2} and \eqref{projex1} respectively.

\item[Step 2] Terminate the algorithm with output $X^*\approx (h^k,\,Q^k)$ if the relative gap
$$
{\left\|(h^{k+1},\,Q^{k+1})-(h^{k},\,Q^k)\right\|_E\over \left\|(h^k,\,Q^k)\right\|_E}~\leq~\epsilon
$$
where $\epsilon\in\mathbb{R}_{++}$ is a prescribed termination threshold. Otherwise, set $k=k+1$ and repeat Step 1 and Step 2. 
 
 \end{description}
 \end{framed}

Convergence of the projection method typically involves continuity and monotonicity of the principal operator which, in our case, is $\mathcal{F}(\cdot)$. This further requires the Lipschitz continuity and monotonicity of the path delay operator $\Psi$. The continuity of the delay operator has been established in a number of references as we noted in Section \ref{secedueexistence}. The monotonicity of the delay operator, on the other hand, may not hold for general networks and traffic flow models; see \cite{MS2007} for a counter example. There are, however, a few studies that show the monotonicity of $\Psi$ under specific assumptions; the reader is referred to  \cite{Mounce} and \cite{PR} for further details. Since the convergence of the projection method for continuous and monotone delay operator is well known, and a generalization of the proof to the elastic demand case is straightforward, we will not elaborate these convergence conditions and repeat the proof in this paper; the reader is referred to \cite{FP}, \cite{Mounce}, \cite{MC}, and \cite{FHNMY} for more details.

Instead, this paper proposes a new set of convergence conditions based on the property of {\it mixed strongly-weakly monotone} (MSWM) of the path delay operator. The MSWM assumes that a subset of the components (paths) of the operator $\Psi$ have the strongly monotone property, and the rest of the components need only be weakly monotone. A weakly monotone operator $M(x)$ is such that
\begin{equation}\label{weakmonotonedef}
\left<M(x_1)-M(x_2),\,x_1-x_2\right>~\geq~-K\left\|x_1-x_2\right\|^2\qquad\forall x_1,\,x_2 \qquad\hbox{for some } ~K~>~0
\end{equation}
The following lemma asserts that all Lipschitz continuous operators are weakly monotone, thereby showing the generality of the weak monotonicity condition. 
\begin{lemma}\label{lemmawm}
Let $U(\cdot)$ be a Lipschitz continuous map defined on a subset $\Omega$ of a topological vector space $V$, that is, there exists $L>0$ such that
$$
\left\|U(x_1)-U(x_2)\right\|~\leq~L\left\|x_1-x_2\right\| \qquad\forall x_1,\,x_2\in \Omega
$$ 
Then $U(\cdot)$ is weakly monotone on $\Omega$. 
\end{lemma}
\begin{proof}
According to the Cauchy-Schwarz inequality, we have
\begin{align*}
\left<U(x_1)-U(x_2),\,x_1-x_2\right>~\geq~-\left\|U(x_1)-U(x_2)\right\|\cdot\left\|x_1-x_2\right\|~\geq~-L\left\| x_1-x_2\right\|^2
\end{align*}
\end{proof}

\begin{remark}
The Lipschitz continuity employed by Lemma \ref{lemmawm} may be further relaxed to continuity if the set $\Omega$ is bounded. 
\end{remark}

\begin{theorem}\label{convthmmswm}{\bf (Convergence of the projection algorithm with the MSWM property)}
Assume that the effective delay operator $\Psi$ and the inverse demand function $\Theta$ are Lipschitz continuous with constants $L_1$ and $L_2$ respectively. In addition, assume the operator $\Psi$ satisfies the MSWM conditions; i.e., there exists a non-empty set $\mathcal{P}^{sm}\subset\mathcal{P}$ (`$sm$' stands for strongly monotone) such that the operator is strongly monotone on these subset of paths:
\begin{equation}\label{MSWMstrongcond}
\sum_{p\in\mathcal{P}^{sm}}\int_{t_0}^{t_f} \big(\Psi_p(t,\,h^{k+1})-\Psi_p(t,\,h^k)\big)\cdot\big(h^{k+1}_p(t) -h^k_p(t)\big)\,dt  ~\geq~K^{sm}\sum_{p\in\mathcal{P}^{sm}} \left\|h^{k+1}_p-h^k_p\right\|^2\quad\forall k~\geq~0
\end{equation}
for some $K^{sm}>0$. For the rest of the paths, the weak monotonicity holds:
\begin{equation}\label{MSWMweakcond}
\sum_{q\in\mathcal{P}\setminus\mathcal{P}^{sm}}\int_{t_0}^{t_f}\big(\Psi_q(t,\, h^{k+1})-\Psi_q(t,\,h^k)\big)\cdot\big(h^{k+1}_q(t) - h^k_q(t) \big)\,dt~\geq~ -K^{wm} \sum_{q\in\mathcal{P}\setminus\mathcal{P}^{sm}}\left\|h^{k+1}_q-h^k_q\right\|^2 \quad\forall k~\geq~0 
\end{equation}
for some $K^{wm}>0$ (`$wm$' stands for weakly monotone). In addition, assume there exists a constant $M>0$ such that
\begin{equation}\label{MSWMcond1}
M \sum_{p\in\mathcal{P}^{sm}}\left\|h_p^{k+1}-h_p^k\right\|^2~\geq~\sum_{q\in\mathcal{P}\setminus\mathcal{P}^{sm}}\left\|h_q^{k+1}-h_q^k\right\|^2
\end{equation}
\noindent Finally, assume that the function $\Theta^-\doteq-\Theta$ is strongly monotone with constant $K_2>0$. Then the sequence $\{h^k\}$ generated by the projection algorithm converges to the solution of the VI \eqref{viextend} provided that 
$$
K^{sm}-K^{wm}M~>~0
$$
\end{theorem}
\begin{proof}
The proof is postponed until \ref{secappmswmthm}.
\end{proof}

\begin{remark}
According to Theorem \ref{convthmmswm}, a necessary condition for the convergence of the projection algorithm is $K^{sm}-K^{wm}M>0$. In other words, the degree of strong monotonicity associated with the subset $\mathcal{P}^{sm}$ must dominate the degree of weak monotonicity associated with the rest of the components. It is quite desirable to identify a set of such paths in a given network based on network topology, path characteristics etc., to meet the sufficient condition for the projection algorithm. Moreover, while the MSWM property may not hold universally for all points in the feasible set $\widetilde\Lambda$, it is more likely to be satisfied at a local level, especially when the trajectory created by the projection algorithm enters certain region. 

Furthermore, the strong monotonicity condition on $\Theta^-=-\Theta$ is satisfied by many inverse demand functions considered in the literature, which are mostly monotonically decreasing (see Table \ref{networkinfo} for an example).
\end{remark}

\subsection{The self-adaptive projection method}\label{subsecsaprojection}

The second computational method is a self-adaptive projection method proposed originally by \cite{HL2002} for generic variational inequalities. Such a method is applied here for computing E-DUEs based on the variational inequality formulation \eqref{viextend}. As we shall see, this method relies on pseudo monotonicity of the delay operator in order to converge.

We begin with some basic notations required to articulate the self-adaptive projection method. As before, $P_{\widetilde\Lambda}[\cdot]$ denotes the projection onto the set $\widetilde\Lambda$. We define the {\it residual}: 
\begin{equation}\label{sapresidual}
r(X; \, \beta)~\doteq~X-P_{\widetilde\Lambda}\big[X-\beta\,\mathcal{F}(X)\big]\qquad X\in\widetilde\Lambda, \quad \beta~>~0
\end{equation}
\noindent where the projection is explicitly given by Theorem \ref{thmexplicitproj}. Notice that the residual is zero if and only if $X$ is a solution of the VI. Given $\alpha,\,\beta>0$, let 
\begin{align}\label{sapd}
d(X;\, \alpha,\beta)~\doteq~&\alpha r(X; \, \beta)+\beta\mathcal{F}\big(X-\alpha r(X; \,\beta)\big)
\\
\label{sapg}
g(X; \, \alpha, \beta)~\doteq~&\alpha\left[r(X;\, \beta) -\beta\Big(\mathcal{F}(X)-\mathcal{F}(X-\alpha r(X; \,\beta))\Big)\right]
\\
\rho(X;\, \alpha, \beta)~\doteq~&{\left<r(X; \,\beta)~,~ g(X; \,\alpha, \beta)\right>_{E}\over \big\|d(X; \,\alpha, \beta)\big\|_{E}^2}
\end{align}

\begin{framed}
\noindent {\bf Self-adaptive projection algorithm}
\begin{description} 

\item[Step 0] Choose fixed parameters $\mu\in (0,\,1),\,\gamma\in(0,\,2),\,\theta>1$, and $L\in(0,\,1)$. Let $\epsilon>0$ be the termination threshold. Identify an initial feasible point $X^0=(h^0,\,Q^0)\in\widetilde\Lambda$ and set iteration counter $k=0$. Let $\alpha_k=1$.

\item[Step 1] Set $\beta_k=\min\{1,\,\theta \alpha_k\}$. Compute the residual $r(X^k; \,\beta_k)$ according to \eqref{sapresidual}. If $\left\|r(X^k; \,\beta_k)\right\|_{E}/ \left\|X^k\right\|_E \leq \epsilon$, terminate the algorithm; otherwise, continue to Step 2.

\item[Step 2] Find the smallest non-negative integer $m_k$ such that $\alpha_{k+1}=\beta_k \mu^{m_k}$ satisfies
\begin{equation}\label{algfindint}
\beta_k\left\|\mathcal{F}(X^k)-\mathcal{F}\big(X^k-\alpha_{k+1}r(X^k;\, \beta_k)\big)\right\|_{E}~\leq~L\left\|r(X^k; \,\beta_k)\right\|_{E}
\end{equation}

\item[Step 3] Compute 
\begin{equation}
X^{k+1}~=~P_{\widetilde\Lambda}\big[ X^k-\gamma\rho(X^k; \,\alpha_{k+1}, \beta_k)d(X^k; \, \alpha_{k+1}, \beta_k)  \big]
\end{equation}
Set $k=k+1$ and go to Step 1.

 \end{description}
 \end{framed}

\eqref{algfindint} requires evaluation of $\mathcal{F}$ at the point $X^k-\alpha_{k+1} r(X^k; \,\beta_k)$. We show that this point always belongs to $\widetilde\Lambda$, the domain of $\mathcal{F}$. Notice that $r(X^k;\, \beta_k)=X^k-P_{\widetilde\Lambda}[X^k-\beta_k \mathcal{F}(X^k)]$, thus 
$$
X^k-\alpha_{k+1} r(X^k; \beta_k)~=~(1-\alpha_{k+1})X^k+\alpha_{k+1}P_{\widetilde\Lambda}[X^k-\beta_k \mathcal{F}(X^k)]\in \widetilde\Lambda
$$
\noindent since both $X^k$ and $P_{\widetilde\Lambda}[X^k-\beta_k\mathcal{F}(X^k)]$ belong to the convex set $\widetilde\Lambda$.

In Step 2 of the algorithm, one is required to test a range of integers, starting from zero, in order to find the smallest integer $m_k$. We show below that such a procedure can always terminate within finite number of trials. Assume that $\mathcal{F}$ is a continuous operator \footnote{Continuity of the operator $\mathcal{F}$ is stated as one of the conditions for solution existence and algorithm convergence.} and observe that $\alpha_{k+1} \to 0$ as $m_k \to +\infty$. There must exist $N>0$ such that for every $m_k>N$ there holds
$$
\left\|\mathcal{F}(X^k)-\mathcal{F}(X^k-\alpha_{k+1}r(X^k; \beta_k))\right\|_{E}~\leq~{L\epsilon\over \beta_k}~\leq~{L\left\|r(X^k; \,\beta_k)\right\|_{E}\over \beta_k}
$$
\noindent which is \eqref{algfindint}. In case $m_k>1$, the algorithm requires more than one evaluation of the operator (that is, more than one dynamic network loading procedure) within one iteration, which is less efficient than the projection algorithm. However, as we  subsequently show, convergence of such an algorithm relies on a weaker notion than strong monotonicity, which the projection algorithm requires.

\begin{definition}{\bf (Pseudo monotone)}\label{defpseudo}
The operator $\mathcal{F}$ is pseudo monotone if, for arbitrary $X^1,\,X^2\in\widetilde\Lambda$, the following holds:
\begin{equation}\label{Phipseudodef}
\left<\mathcal{F}(X^2),\,X^1-X^2\right>_{E}~\geq~0 \quad \hbox{implies}\quad \left<\mathcal{F}(X^1),\,X^1-X^2\right>_{E}~\geq~0
\end{equation}
\end{definition}

\noindent By definition, pseudo monotonicity is a consequence of monotonicity, and thus is one type of generalized monotonicity \citep{PS1997}.  The convergence of the proposed projection algorithm requires the following property of the principal operator:
\begin{equation}\label{Phiconvcond}
\left<\mathcal{F}(X),\, X-X^*\right>_{E}~\geq~0\qquad\forall X\in \widetilde\Lambda
\end{equation}
\noindent where $X^*$ is a solution of the original VI \eqref{viextend}. Notice that \eqref{Phiconvcond} follows from monotonicity or pseudo monotonicity of $\mathcal{F}$, and thus is weaker than these two monotonicity conditions.  The following convergence proof is due to \cite{HL2002}.

\begin{theorem}\label{thmprojalgconv}{\bf (Convergence of the self-adaptive projection method)}. Assume that $\mathcal{F}:\,\widetilde\Lambda\to E$ is continuous and satisfies \eqref{Phiconvcond}. Then the sequence $\{X^k\}$ generated by the self-adaptive projection algorithm converges to a solution of the VI \eqref{viextend}.
\end{theorem}
\begin{proof}
See \cite{HL2002} for a proof. 
\end{proof}

\subsection{The proximal point method}

The {\it proximal point method} (PPM) \citep{Konnov} is a popular method for solving optimization problems and variational inequalities. It replaces the original problem with a sequence of regularized problems, each of which can be solved with standard algorithms due to improved regularity. The PPM is known to converge with some generalized monotonicity \citep{Allevi}. In this paper we apply the PPM to solve E-DUE problems while advancing our knowledge of the PPM by proposing a set of new convergence conditions weaker than those from previous studies.  The proximal point method has been further developed in this paper and is summarized below.

\begin{framed}
\noindent {\bf Proximal point method}
\begin{description} 

\item[Step 0] Identify an initial feasible point $X^0=(h^0,\,Q^0)\in\widetilde\Lambda$. Fix a large constant $a>0$ and set a tolerance parameter $\delta>0$. Set the iteration counter $k=0$. 

\item[Step 1] Solve the following variational inequality for $X^{k+1}=(h^{k+1},\,Q^{k+1})$: 
\begin{equation}\label{ppmvi}
\left<\mathcal{F}(X^{k+1}) + a(X^{k+1}-X^k)~,~  X - X^{k+1}\right>_E~\geq~0\qquad\forall X\in\widetilde\Lambda
\end{equation}

\item[Step 2] Terminate the algorithm if $\|X^{k+1}-X^k\|_E \leq {\delta\over a D}$, where $D$ is the diameter of the set $\widetilde\Lambda$.  Otherwise, set $k=k+1$ and repeat Step 1 through Step 2.

 \end{description}
 \end{framed}

\noindent The key step of the PPM is to solve the VI \eqref{ppmvi}, which  enjoys a significantly improved regularity than the original VI problem. To see this, we rewrite $\mathcal{F}(X^{k+1}) + a(X^{k+1}-X^k)$  as $(\mathcal{F}+a I)(X^{k+1}) -a X^k$, where $I$ is the identity map. If $\mathcal{F}$ is weakly monotone with constant $-K$ (see \eqref{weakmonotonedef}), then $(\mathcal{F}+a I)(X^{k+1}) -a X^k$ is a strongly monotone operator acting on $X^{k+1}$ provided that $a>K$. Thus, by choosing $a$ large enough, the VI \eqref{ppmvi} can be solved with any existing algorithm with satisfactory convergence result.

We now present a new convergence result for the proximal point method. We begin with the articulation of the {\it dual formulation} of the VI \eqref{viextend}. 
\begin{definition}{\bf (Dual formulation of the VI)}
The dual form of the VI \eqref{viextend}, also known as the Minty problem, is defined as follows. Find $X^d\in\widetilde\Lambda$ such that 
\begin{equation}\label{minty}
\left<\mathcal{F}(X),\, X-X^d\right>_{E}~\geq~0\qquad\forall X\in\widetilde\Lambda
\end{equation}
We let $Y^d$ be the solution set of \eqref{minty}. 
\end{definition}

\begin{lemma}\label{ppmconvthm}
Assume that $Y^d\neq\emptyset$, and that $\widetilde\Lambda$ is bounded with diameter $D<\infty$.  Then the sequence $\{X^k\}$ generated by the proximal point method satisfies the following: for any $k\geq 1$, 
\begin{equation}\label{Liueqn1}
\left<\mathcal{F}(X^{\mu(k)+1})~,~ X- X^{\mu(k)+1}\right>_E~\geq~ - {a D^2\over \sqrt{k+1}} \qquad\forall X\in\widetilde\Lambda
\end{equation}
where $\mu(k)\doteq \underset{0\leq i\leq k}{\text{argmin}}\,\big\|X^i -X^{i+1}\big\|_E^2 ~\in~\{0,\,\ldots,\, k\}$.   
\end{lemma}
\begin{proof}
The proof is postponed until \ref{secappppm}.
\end{proof}

\begin{remark}
We note that the feasible set $\widetilde\Lambda$ defined originally for continuous-time functions is unbounded. However, in numerical (discrete-time) computations the feasible set $\widetilde\Lambda$ will be bounded with a finite diameter $D<\infty$.
\end{remark}

As an immediate consequence of Lemma \ref{ppmconvthm}, we have the following convergence result.

\begin{theorem}{\bf (Convergence of the proximal point method)}\label{ppmthm}
Assume that $Y^d\neq \emptyset$, and that $\widetilde\Lambda$ is bounded with diameter $D<\infty$. Then for any tolerance $\delta>0$, there exists $R\doteq \lceil{a^2D^4\over \delta^2}\rceil-1$, such that 
\begin{equation}\label{ppmconveqn1}
\left<\mathcal{F}(X^{\mu(R)+1})~,~X-X^{\mu(R)+1}\right>_E~\geq~-\delta \qquad \forall X\in\widetilde\Lambda,
\end{equation}
where $\{X^k\},\,k\geq 0$ is the sequence generated by the proximal point method, $a$ is the number appearing in \eqref{ppmvi}, and $\lceil z\rceil$ denotes the smallest integer that is larger than or equal to $z$. Moreover, when the PPM algorithm terminates, i.e. when $\|X^{k+1}-X^k\|_E\leq {\delta\over aD}$ for the first time, then
\begin{equation}\label{ppmconveqn2}
\left<\mathcal{F}(X^{k+1})~,~X-X^{k+1}\right>_E~\geq~-\delta \qquad\forall X\in\widetilde\Lambda
\end{equation}
\end{theorem}
\begin{proof}
Replacing $k$ in \eqref{Liueqn1} with $R=\lceil{a^2D^4\over \delta^2}\rceil-1$ yields
$$
\left<\mathcal{F}(X^{\mu(R)+1})~,~X-X^{\mu(R)+1}\right>_E~\geq~ -\delta \qquad\forall X\in\widetilde\Lambda
$$
\noindent which is \eqref{ppmconveqn1}. To show \eqref{ppmconveqn2}, we invoke \eqref{Liueqn7} in the proof of Lemma \ref{ppmconvthm}:
$$
\left<\mathcal{F}(X^{\mu(k)+1})~,~X-X^{\mu(k)+1}\right>_E~\geq~-a D \left\|X^{\mu(k)+1}-X^{\mu(k)}\right\|_E
$$
\noindent According to the definition of $\mu(\cdot)$, the termination criterion of the PPM implies that $k=\mu(k)$. Thus we have that, when the algorithm terminates,
$$
\left<\mathcal{F}(X^{k+1})~,~X-X^{k+1}\right>_E~\geq~-a D \left\|X^{k+1}-X^{k}\right\|_E~\geq~-\delta
$$
\end{proof}

\begin{remark}
Unlike the convergence results established in Section \ref{subsecprojection} and \ref{subsecsaprojection}, which focus on the asymptotic behavior of the sequence $\{X^k\}$ as $k\to\infty$,  the convergence result developed in Theorem \ref{ppmthm} is concerned with finding a solution of the approximate VI (that is, with $-\delta$ on the right hand side instead of zero) within finite iteration. Such a convergence result is quite practical for numerical computations as it estimates the number of iterations needed to achieve certain level of approximation of the original VI.
\end{remark}

Theorem \ref{ppmthm} only requires that the dual VI \eqref{minty} has a solution -- a property subsequently referred to as {\it dual solvability}. Compared to the convergence conditions for the self-adaptive projection method (Section \ref{subsecsaprojection}), dual solvability is weaker than the assumption \eqref{Phiconvcond} as the latter requires that the solution of the original VI must be a solution of the dual VI. In addition, Theorem \ref{ppmthm} does not rely on the continuity of the principal operator. We thus conclude that the convergence conditions for the PPM are indeed weaker than the previous two methods.

In the remainder of this section, we will investigate in detail the dual solvability and provide  sufficient conditions for it. One should note that if the original VI has a solution, then a sufficient condition for dual solvability is pseudo monotonicity; this is apparent from Definition \ref{defpseudo}. In the following presentation, we will articulate a weaker  sufficient condition for dual solvability, based on the notion of semistrictly quasi monotonicity. 

\begin{definition}
The operator $\mathcal{F}$ is quasi monotone if, for arbitrary $X^1,\,X^2\in\widetilde\Lambda$, 
$$
\left<\mathcal{F}(X^2)~,~X^1-X^2\right>_E>0~\Longrightarrow~\left<\mathcal{F}(X^1)~,~X^1-X^2\right>_E\geq0
$$
The operator $\mathcal{F}$ is semistrictly quasi monotone if it is quasi monotone and, for every $X^1,\,X^2\in \widetilde\Lambda$,  
$$
\left<\mathcal{F}(X^2)~,~X^1-X^2\right>_E>0~\Longrightarrow~\left<\mathcal{F}(X^3)~,~X^1-X^2\right>_E>0 
$$
for some $X^3\in \Big\{X:~X=X^1+\lambda(X^2-X^1),~~\lambda\in(0,\,{1\over 2}) \Big\}$.
\end{definition}

\noindent The reader is referred to \cite{Konnov1998} for a detailed discussion of quasi monotonicity. In particular, Lemma 3.1 of \cite{Konnov1998} states that pseudo monotonicity implies semistrictly quasi monotonicity, making the latter a weaker assumption. The notion of $w^*$-hemicontinuity is also needed for the dual solvability. 
\begin{definition}
$\mathcal{F}$ is $w^*$-hemicontinuous if the function 
$$
c(\lambda)~\doteq~\left<\mathcal{F}(X_{\lambda})~,~X_{\lambda} - X^2\right>_E\qquad \hbox{where }~X_{\lambda}=\lambda X^1+(1-\lambda)X^3
$$
is upper semicontinuous at $\lambda=0+$ for all $X^1,\,X^2,\,X^3\in\widetilde\Lambda$ and $\lambda\in[0,\,1]$.
\end{definition}

\noindent Finally, the sufficient condition for dual solvability is summarized below. 
\begin{theorem}{\bf (Sufficient condition for dual solvability)}
If $\mathcal{F}$ is continuous on $\widetilde\Lambda$ and is semistrictly quasi monotone, then the dual problem \eqref{minty} has a solution.
\end{theorem}
\begin{proof}
It is easy to verify by definition that if $\mathcal{F}$ is continuous in the strong topology, then it is $w^*$-hemicontinuous. Thus the conclusion follows from Theorem 4.1 of \cite{Konnov1998}.
\end{proof}

\begin{remark}
The three computational methods  proposed here rely on generalized monotonicity in order to converge. As we previously mentioned, most delay operators may not satisfy these generalized notions of monotonicity, with only a few exceptions \citep{Mounce, PR}. Thus the proper perspective on convergence of numerical algorithms for calculating DUEs on general networks is to say that almost all algorithms are presently heuristic. Exact algorithms will be created only when a fundamentally new operator class is invented, which allows non-monotonicity while also providing behavioral insights that allow convergence to be established.
\end{remark}

\section{Numerical examples}\label{secnum}

In this section we numerically test the three computational algorithms and illustrate their solutions. Two test networks are considered: the seven-arc network and the Sioux Falls network as shown in Figure \ref{FigTestnetworks}. Table \ref{networkinfo} shows the number of O-D pairs and paths in the test networks, as well as the inverse demand functions, which are assumed to be the same for all the O-D pairs. In particular, for the seven-arc network one O-D pair $(1,\,6)$ and three paths $p_1=\{I_1,\,I_3,\,I_6,\,I_7\},\,p_2=\{I_1,\,I_2,\,I_4,\,I_6,\,I_7\}, p_3=\{I_1,\,I_2,\,I_5,\,I_7\}$ are considered. In the Sioux Falls network six O-D pairs $(1,\,20),\,(2,\,20),\,(3,\,20),\,(4,\,20),\,(5,\,20)$ and $(6,\,20)$ are considered, among which 119 paths are utilized.

\begin{figure}[h!]
\centering
\includegraphics[width=0.7\textwidth]{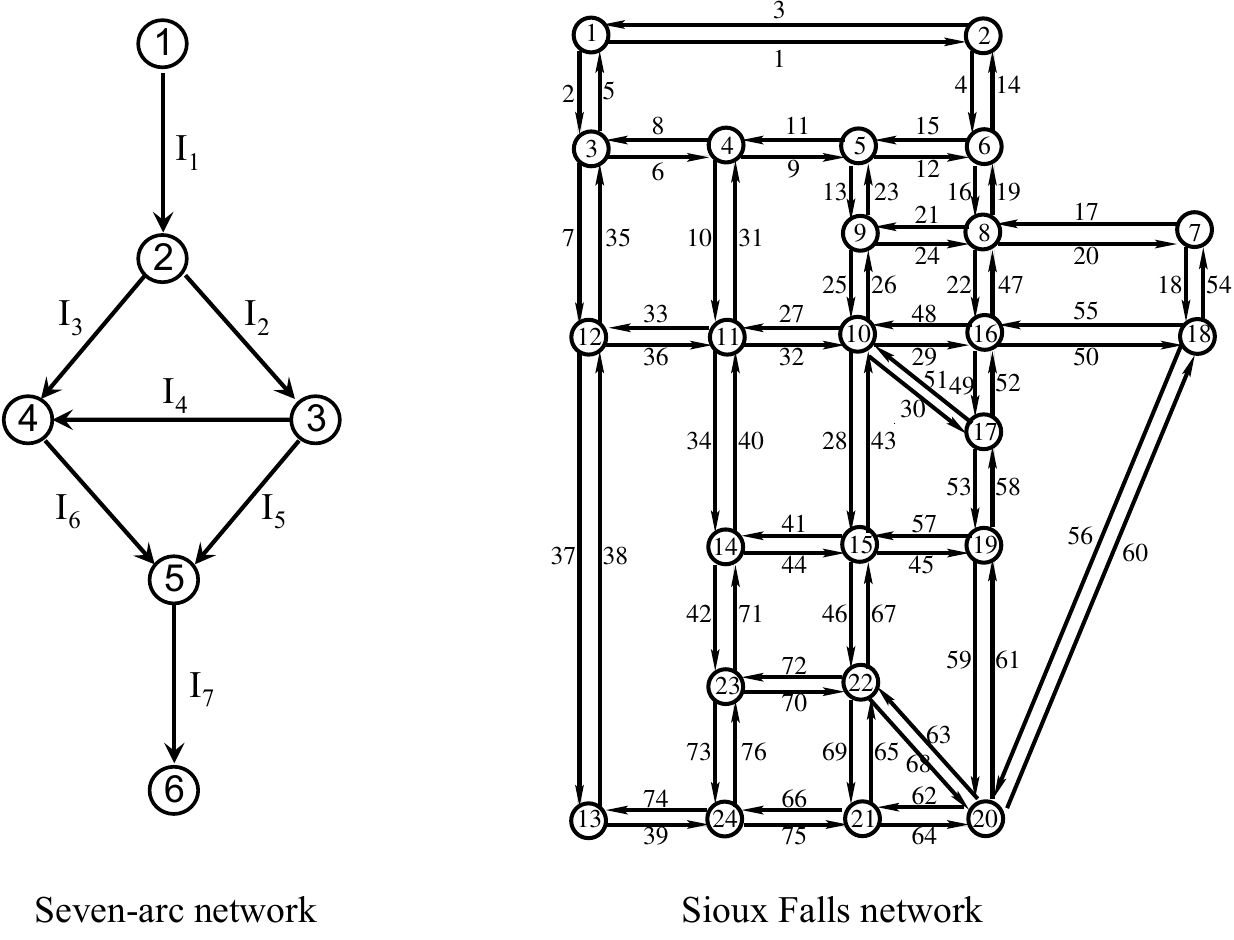}
\caption{The two test networks}
\label{FigTestnetworks}
\end{figure}

\begin{table}[h!]
\center
\begin{tabular}{|c|c|c|c|}
\hline
          &  \# of O-D pairs   &  \# of paths   & Inverse demand function \\\hline
Seven-arc network &   1   &  3   &   $v=-Q/ 2000  + 1.2$     \\\hline
Sioux Falls network &  6   &  119   &  $v=-Q/ 500+1.6$  \\\hline
\end{tabular}
\caption{Basic network information. The same inverse demand function is assumed for all the O-D pairs in the network.}
\label{networkinfo}
\end{table}

The following termination criteria are employed for the three computational algorithms proposed in this paper. For the projection method and the proximal point method, the algorithm is terminated if the relative gap 
\begin{equation}\label{termination1}
{\| X^{k+1} - X^k \|_E \over \| X^k \|_E} ~\leq~10^{-5}
\end{equation}
where $X^k\doteq (h^k,\,Q^k)$. For the self-adaptive projection method, the termination criterion is 
\begin{equation}\label{termination2}
{\| r(X^k;\,\beta_k)\|_E\over \|X^k\|_E}~\leq~ 10^{-6}
\end{equation}
where the residual $r(X^k;\,\beta_k)$ is given in \eqref{sapresidual}. These termination criteria will be adopted for the computations on both test networks.

Regarding the dynamic network loading sub-problem, we employ the link transmission model \citep{LTM}, which is a discrete-time and simplified version of the Lighthill-Whitham-Richards model \citep{LW, Richards}. It is based on a triangular fundamental diagram and propagates traffic on a link through Newell's variational principle \citep{Newella}. Due to space limitation we will not elaborate this DNL procedure here and instead refer the reader to \cite{LTM} for the original model and \cite{LKWM} for more elaborated discussion. All the computations reported here were performed on a standard laptop with 4 GB of RAM.

\subsection{Computational results on the severn-arc network}

We apply the projection method (PM), the self-adaptive projection method (SA), and the proximal point method (PP) to solve the elastic DUE problem on the seven-arc network shown in Figure \ref{FigTestnetworks}. Under the termination criteria \eqref{termination1}-\eqref{termination2}, the algorithms converge after 200 (PM), 3903 (SA), and 180 (PP) iterations. We use Figure \ref{FigSeven_conv} to show the left hand sides of \eqref{termination1} and \eqref{termination2} at each iteration of the algorithms. It can be seen that although all three algorithms reach the termination threshold within finite number of iterations, their convergence speeds vary.  In particular, the self-adaptive projection method has a much slower yet more smooth convergence than the other two; and its relative gap is monotonically decreasing as the iteration continues. The projection method and the proximal point method show qualitatively similar convergence trends: their relative gaps decrease much more quickly than the self-adaptive method but may have some local increase and oscillation.

\begin{figure}[h!]
\centering
\includegraphics[width=\textwidth]{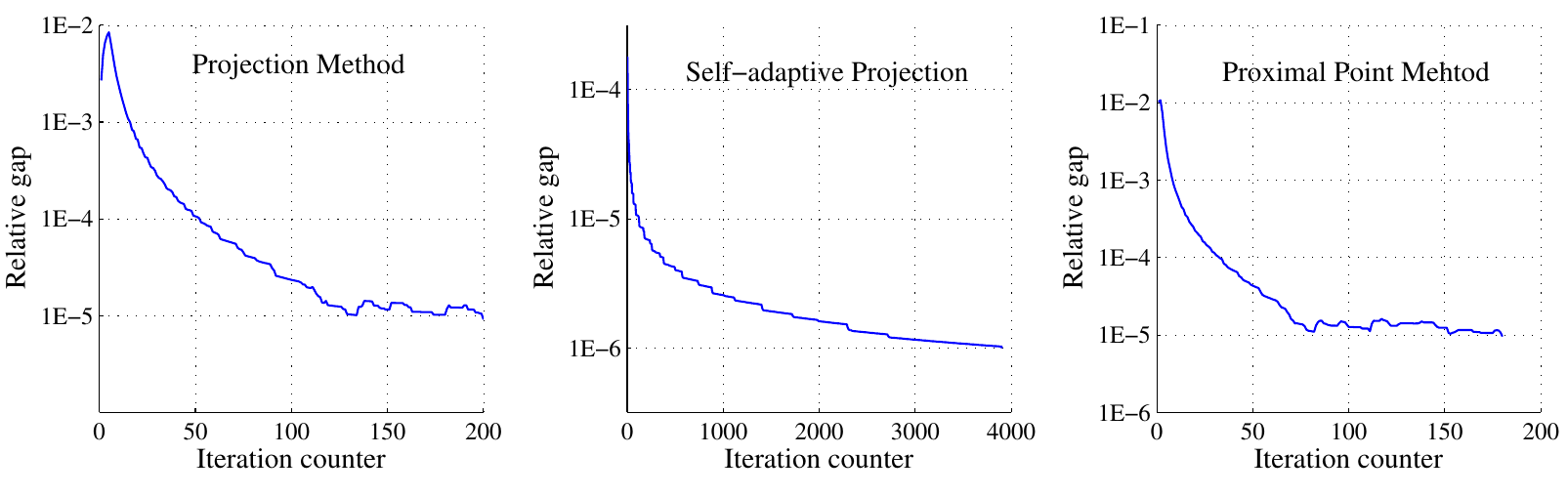}
\caption{The seven-arc network: Relative gaps, defined in \eqref{termination1}-\eqref{termination2}, at each iteration of the algorithms. Here, 1E-x means $10^{-x}$.}
\label{FigSeven_conv}
\end{figure}

We now turn to the solutions produces by these algorithms. Figure \ref{FigSeven_Qnu} illustrates two main quantities: the elastic demand $Q^k$ generated at each iteration, and the drivers' average travel cost at each iteration defined as: 
$$
v^k~\doteq~{1\over Q^k} \sum_{p\in\mathcal{P}} \int_{t_0}^{t_f}\Psi_p(t,\,h^k)\cdot h_p^k(t)\,dt \qquad k~=~0,\,1,\, 2,\,\ldots
$$

\begin{figure}[h!]
\centering
\includegraphics[width=\textwidth]{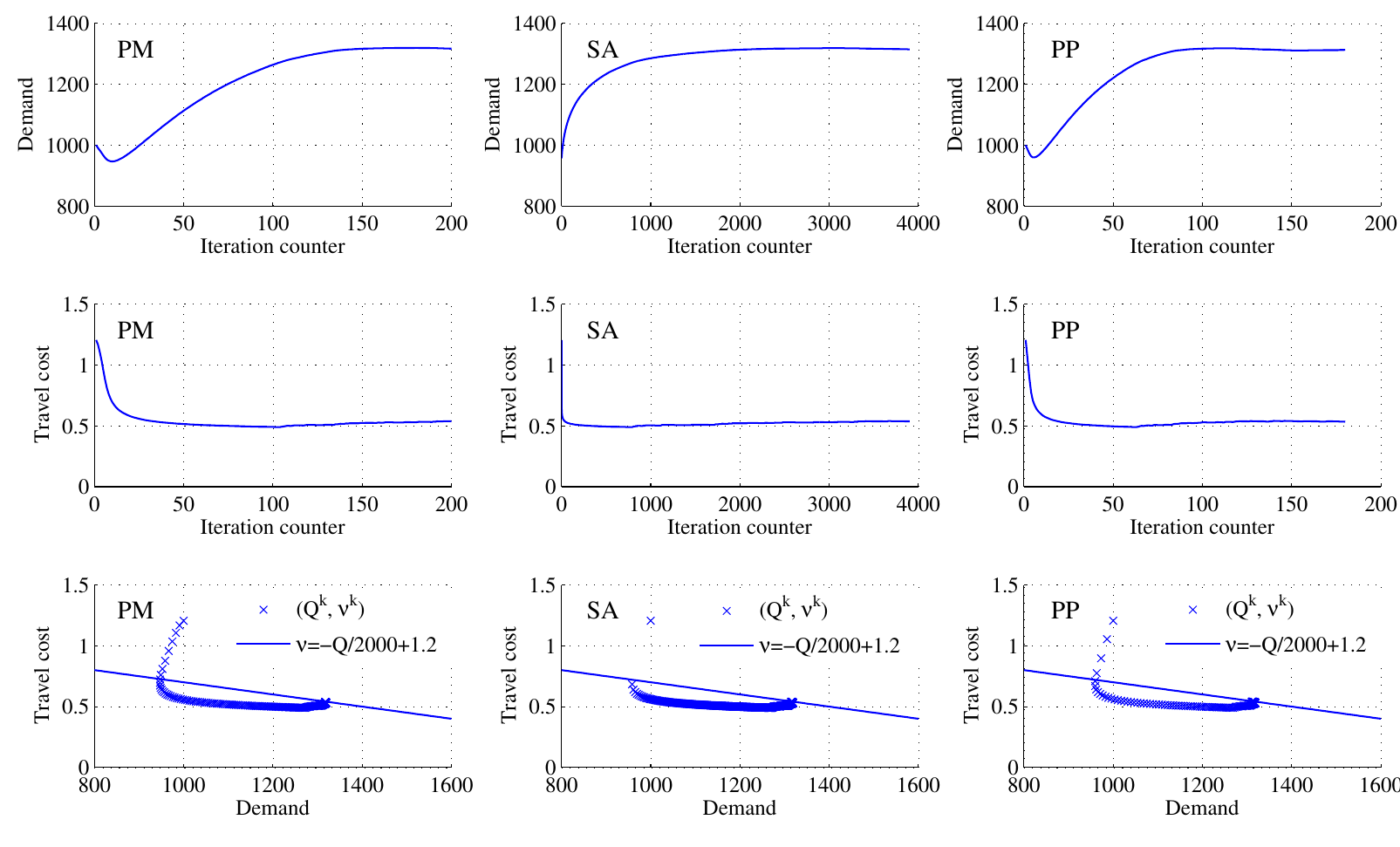}
\caption{The seven-arc network: The elastic demand $Q^k$ and the average travel cost $v^k$ at each iteration. The initial value for the demand is always set to be $Q^0=1000$. PM: projection method; SA: self-adaptive projection; PP: proximal point.}
\label{FigSeven_Qnu}
\end{figure}

\noindent The reason for using the average travel cost $v^k$ is as follows. Until an E-DUE has been found (as allowed by the prescribed tolerance), the travel costs within the same O-D pair are not equalized; thus we use the averaged travel cost to demonstrate the demand-cost relationship stipulated by the inverse demand function shown in Table \ref{networkinfo}. It can be seen from Figure \ref{FigSeven_Qnu} that both $Q^{k}$ and $v^k$ converge to a fixed value when the algorithms terminate, and they converge to a state that is consistent with the inverse demand function (see the bottom row of Figure \ref{FigSeven_Qnu}).

It remains to show that the experienced travel costs within the same O-D pair for all route and departure time choices are equal and minimum. This is illustrated in Figure \ref{FigSeven_fE}, where we show the path departure rates and the corresponding travel costs for all three paths in the network. We see that the travel costs are indeed equal and minimum whenever the departure rates are positive, which confirm the equilibrium state. Nevertheless, the equal travel costs claimed above are not exact if one inspects the pictures closely, and we use Remark \ref{numericalremark} below to justify our results.

\begin{remark}\label{numericalremark}
When it comes to computation of DUEs numerically, two facts need to be realized. First, theories regarding the continuous-time formulations of E-DUE and the delay operator are exact only when the time step employed by the numerical computation is infinitely small, which is of course impossible. From this perspective, all computational results reported so far, not merely in this paper, are approximations of the DUE or E-DUE. Secondly, most computational algorithms, even when their convergence criteria are met, can only converge in the limit of an infinite number of iterations. In other words, their convergence is in the asymptotic sense. It is safe to say that all computations of DUE or E-DUE are terminated after a finite number of iterations, and thus the solutions are, again, approximations as allowed by prescribed tolerances.
\end{remark}

\begin{figure}[h!]
\centering
\includegraphics[width=\textwidth]{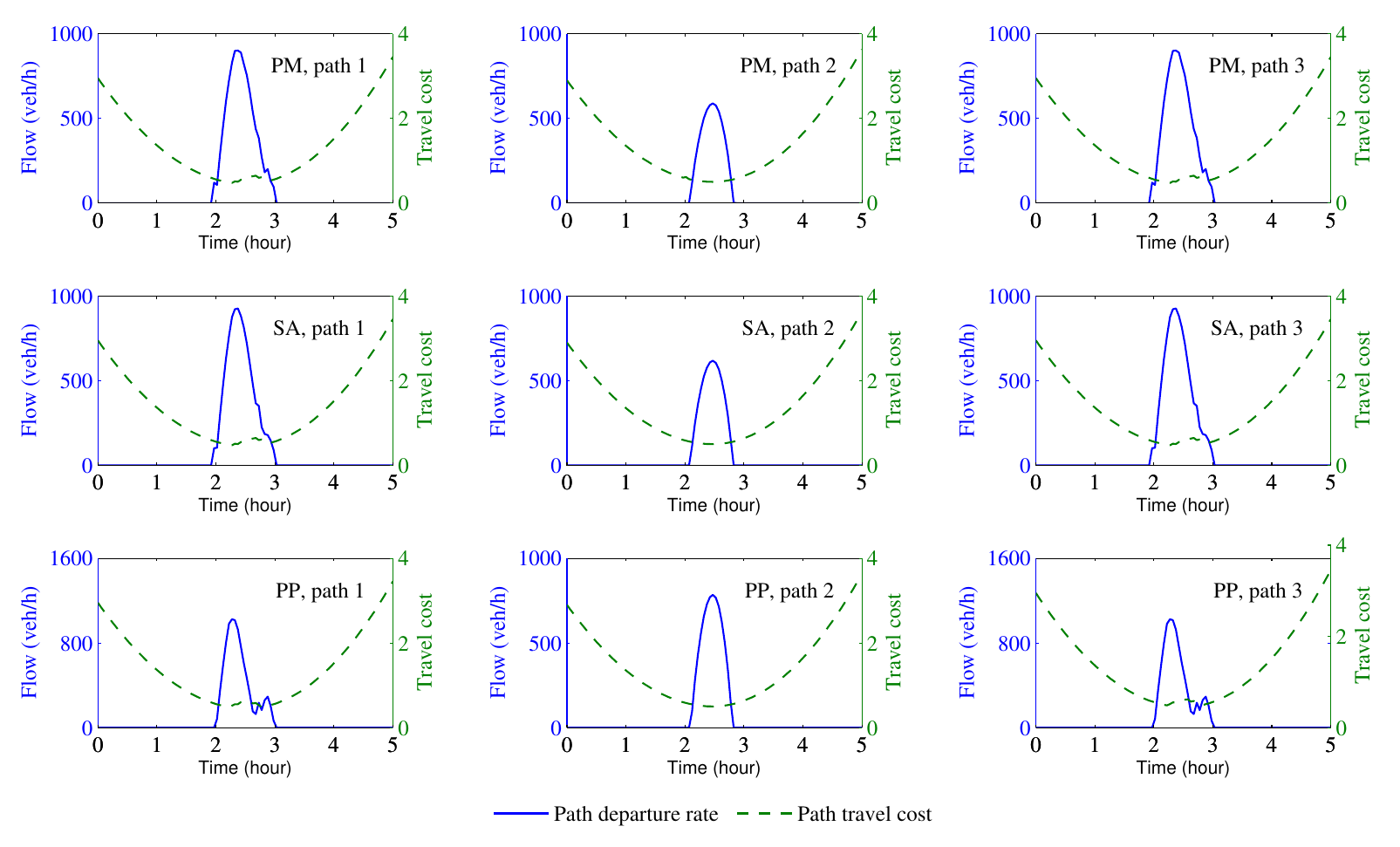}
\caption{The seven-arc network: Path departure rates and associated travel costs for all three paths. PM: projection method; SA: self-adaptive projection; PP: proximal point.}
\label{FigSeven_fE}
\end{figure}

\subsection{Computational results on the Sioux Falls network}

We test the three proposed algorithms further on the larger Sioux Falls network (see Figure \ref{FigTestnetworks}). We start with the algorithm convergence following the same criteria \eqref{termination1}-\eqref{termination2}. Figure \ref{FigSF_conv} shows the relative gap at each iteration of the algorithm. The number of iterations needed for the three algorithms are: 469 (PM), 2304 (SA), and 109 (PP). Observations regarding the convergence trends are similar to the previous test network: The self-adaptive method has a slower yet more smooth convergence rate, and the relative gap is monotonically decreasing. The projection method and the proximal point method have a faster convergence but their relative gaps may have some local oscillations.

\begin{figure}[h!]
\centering
\includegraphics[width=\textwidth]{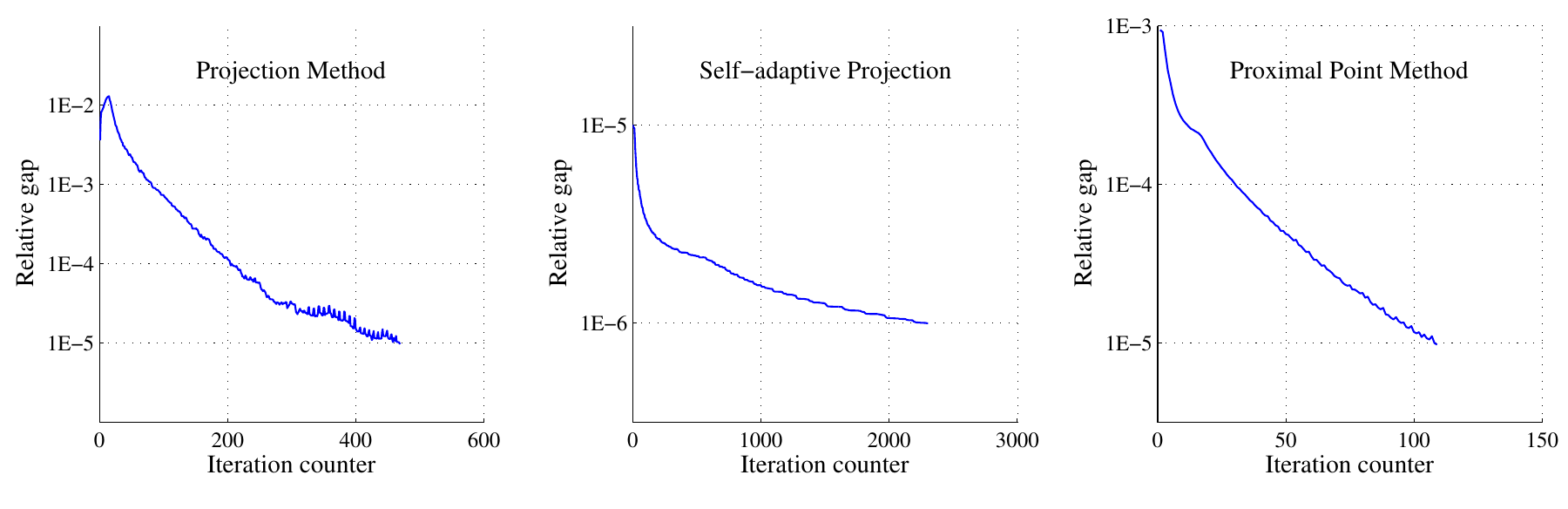}
\caption{The Sioux Falls network: Relative gaps, defined in \eqref{termination1}-\eqref{termination2}, at each iteration of the algorithms. Here, 1E-x means $10^{-x}$.}
\label{FigSF_conv}
\end{figure}

The convergence of the elastic travel demands $Q_{ij}^k$ and the OD-specific average travel costs $v_{ij}^k$ is demonstrated in Figure \ref{FigSF_Qnu}, which has a similar style of presentation as Figure \ref{FigSeven_Qnu} but shows information for all six O-D pairs. We can, again, confirm that $Q^{k}_{ij}$ and $v_{ij}^k$ converge to the relationship stipulated by the inverse demand function for all O-D pairs when the algorithms converge. However, there are some inaccuracies in the convergence of the self-adaptive method and the proximal point method as the trajectories of $(Q^k_{ij},\,v_{ij}^k)$ did not reach the line $v=-Q/500+1.6$ when the algorithms terminated. They are, again, related to the observations made in Remark \ref{numericalremark}, and may be overcome by tighter stopping tolerances.

\begin{figure}[h!]
\centering
\includegraphics[width=\textwidth]{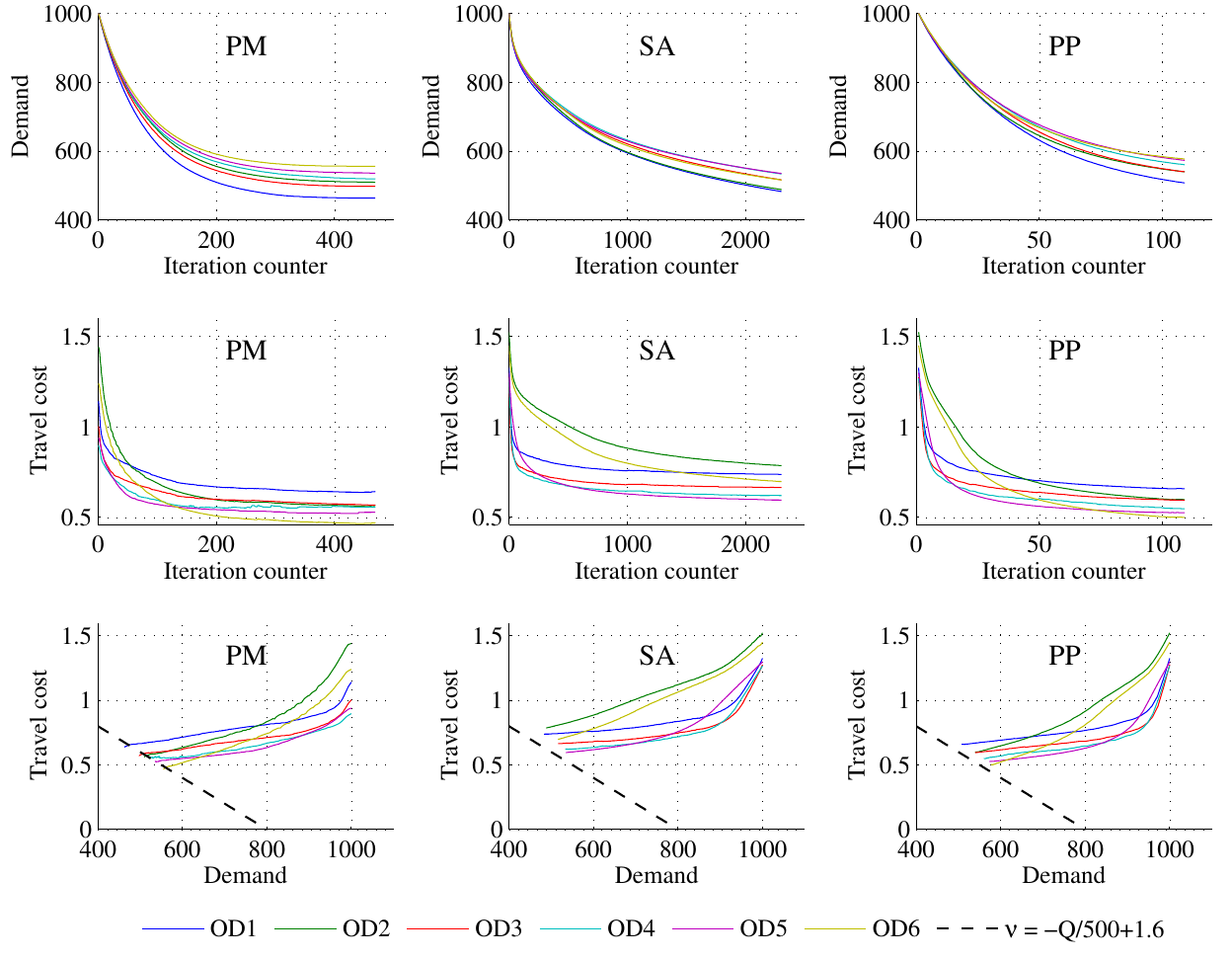}
\caption{The Sioux Falls network: The elastic demands $Q_{ij}^k$ and the average travel costs $v^k_{ij}$ for each O-D pair $(i,\,j)$ at each iteration. The initial values for the demands are always set to be $Q_{ij}^0=1000$ for all $(i,\,j)\in\mathcal{W}$. PM: projection method; SA: self-adaptive projection; PP: proximal point.}
\label{FigSF_Qnu}
\end{figure}

\subsection{Summary of algorithm performances}

We summarize the performances of the three algorithms in terms of solution quality, convergence rate, computational complexity, and computational time. Table \ref{tabperformance} shows some basic information on their computational performances.

\begin{table}[h!]
\begin{center}

\begin{tabular}{c|c|c|c|c|c|c|}\cline{2-7}
 & \multicolumn{3}{|c|}{ Seven-arc}& \multicolumn{3}{|c|}{Sioux Falls}
\\
\cline{2-7}
          & Iteration \# & DNL \# & Time  & Iteration \# & DNL \# & Time
\\
\hline
    \multicolumn{1}{|c|}{Projection Method}  & 200 & 200  & 3 s            & 469 & 469 & 382 s
 \\
\hline

  \multicolumn{1}{|c|}{Self-Adaptive}    & 3903 &  7805  & 108 s            & 2304 & 4,607  & 3810 s
 \\
\hline 
  \multicolumn{1}{|c|}{Proximal Point}      & 180 &  5400  & 74 s            & 109 & 2332  & 2425 s
 \\
\hline
\end{tabular}
\end{center}
\caption{Summary of the three algorithms.}
\label{tabperformance}
\end{table}

In general, the self-adaptive projection method takes much more iterations to reach a prescribed convergence threshold than the other two methods. However, one should take into account the fact that the self-adaptive method utilizes a different termination criterion than the other two; i.e. it relies on the residual function $r(X^k;\,\beta^k)$ while the others focus on the distance between two consecutive iterates (see \eqref{termination1} and \eqref{termination2}). 

The self-adaptive method and the proximal point method require multiple dynamic network loading (DNL) procedures to be performed within a single iteration, thus end up with more DNL numbers than the projection method. In our computation, the average numbers of DNLs performed within one iteration are roughly 2 (self-adaptive) and 26 (proximal point). The proximal point method requires solving a regularized VI within each iteration and thus demands substantially more computational effort.

The projection method requires the least computational effort and yields good solution quality despite its restrictive convergence condition (based on strong monotonicity). The self-adaptive method and the proximal point method require more iterations and DNL procedures to converge, which highlights a trade-off between theoretical convergence and computational complexity. Overall, all three algorithms have successfully computed the E-DUE solutions with convergence as allowed by a prescribed stopping threshold, despite the fact that their convergence conditions have not been verified by the network performance model.

\section{Conclusion}\label{secconclusion}
This paper address three aspects of the simultaneous route-and-departure-time dynamic user equilibrium (SRDT DUE): formulation, existence and computation. This problem is first formulated as a variational inequality problem in an extended Hilbert space. A general existence theory for the continuous-time E-DUE problem is then proposed  based on the new VI formulation. This existence proof employs the most general constraints relating path departure rates to a table of elastic trip volumes, and does not invoke the {\it a priori} upper bounds on the path departure rates.  Finally, we present three new computational algorithms: the projection algorithm, the self-adaptive projection algorithm, and the proximal point method. The first and second methods require a minimum-norm projection onto the extended Hilbert space, which can be explicitly instantiated using the differential variational inequality (DVI) formalism.  Convergence proofs are provided for all three algorithms with different types of generalized monotonicity, namely the mixed strongly-weakly monotonicity, the pseudo monotonicity, and the semistrictly quasi monotonicity. These three algorithms are tested on two networks with their solution quality, convergence, and computational efficiency shown and compared. The numerical results show a clear trade-off between the generality of theoretical convergence and the computational efficiency of the algorithms. In particular, the projection method, despite its restrictive convergence condition, tends to have a faster convergence rate than the other two methods.

It should be noted that the relaxed notions of monotonicity employed by our convergence proofs are not verified against the network performance model, thus the computational methods proposed here should be considered heuristics. Presently known mathematics does not provide a means of
classifying path delay operators like those intrinsic to our formulation of
E-DUE. Rather a \underline{new} class of operators must be discovered. The
first two steps in that process of discovery are: (i) the illustration of
what can be understood about convergence based on current knowledge; and
(ii) the heuristic application of known algorithms when convergence cannot
be rigorously assured. Again, it is our formulations that allow us to offer
results directly relevant to the aforementioned two steps.

\section{Acknowledgement} 
The comments of Professor Alberto Bressan (Penn State University) are gratefully acknowledged.


 \appendix

\section{Proof of Theorem \ref{vielasticthm}}\label{secappVIthm}

\begin{proof}
{\bf [Necessity].} Given a DUE solution with elastic demand $(h^*,\,Q^*)\in\widetilde\Lambda$, we easily deduce from \eqref{chapVI:eqn1} and \eqref{chapVI:eqn2} that for any $(h,\,Q)\in\widetilde\Lambda$, 
\begin{align}
&\sum_{p\in \mathcal{P}}\int\nolimits_{t_{0}}^{t_{f}}\Psi _{p}(t,h^*)(h_{p}(t)-h_{p}^{\ast }(t))dt
-\sum_{(i,\,j) \in \mathcal{W}}\Theta _{ij}\left[ Q^{\ast } \right] (Q_{ij} -Q_{ij}^{*}) \nonumber
\\
~=~& \sum_{(i,\,j)\in\mathcal{W}}\left(\sum_{p\in\mathcal{P}_{ij}}\int_{t_0}^{t_f}\Psi_p(t,\,h^*)h_p(t)\,dt-\Theta_{ij}[Q^*]\cdot Q_{ij}\right)\nonumber
\\
&~-~\sum_{(i,\,j)\in\mathcal{W}}\left(\sum_{p\in\mathcal{P}_{ij}}\int_{t_0}^{t_f}\Psi_p(t,\,h^*)h^*_p(t)\,dt-\Theta_{ij}[Q^*]\cdot Q^*_{ij}\right)\nonumber
\\
~=~&\sum_{(i,\,j)\in\mathcal{W}}\left(\sum_{p\in\mathcal{P}_{ij}}\int_{t_0}^{t_f}\Psi_p(t,\,h^*)h_p(t)\,dt-\Theta_{ij}[Q^*]\cdot Q_{ij}\right)\nonumber
\\
&~-~\sum_{(i,\,j)\in\mathcal{W}}\left(\sum_{p\in\mathcal{P}_{ij}}\Theta_{ij}[Q^*]\cdot\int_{t_0}^{t_f}h^*_p(t)\,dt- \Theta_{ij}[Q^*]\cdot Q^*_{ij}\right)\nonumber
\\
~=~& \sum_{(i,\,j)\in\mathcal{W}}\left(\sum_{p\in\mathcal{P}_{ij}}\int_{t_0}^{t_f}\Psi_p(t,\,h^*)h_p(t)\,dt-\Theta_{ij}[Q^*]\cdot Q_{ij}\right)\nonumber
\\
&~-~\sum_{(i,\,j)\in\mathcal{W}}\Theta_{ij}[Q^*]\cdot\left(\sum_{p\in\mathcal{P}_{ij}}\int_{t_0}^{t_f}h^*_p(t)\,dt- Q^*_{ij}\right)\nonumber
\\
~=~& \sum_{(i,\,j)\in\mathcal{W}}\left(\sum_{p\in\mathcal{P}_{ij}}\int_{t_0}^{t_f}\Psi_p(t,\,h^*)h_p(t)\,dt-\Theta_{ij}[Q^*]\cdot Q_{ij}\right)\label{chapVI:eqn3}
\end{align}
Observe that within \eqref{chapVI:eqn3}, we have, given any $(i,\,j)\in\mathcal{W}$, that
\begin{align}
&\sum_{p\in\mathcal{P}_{ij}}\int_{t_0}^{t_f}\Psi_p(t,\,h^*)h_p(t)\,dt-\Theta_{ij}[Q^*]\cdot Q_{ij}\nonumber\\
~\geq~& v_{ij}(h^*)\sum_{p\in\mathcal{P}_{ij}}\int_{t_0}^{t_f}h_p(t)\,dt-\Theta_{ij}[Q^*]\cdot Q_{ij}~=~v_{ij}(h^*)\left(\sum_{p\in\mathcal{P}_{ij}}\int_{t_0}^{t_f}h_p(t)\,dt-Q_{ij}  \right)~=~0 \label{chapVI:eqn4}
\end{align}
\noindent where $v_{ij}(h^*)$ is the essential infimum of the effective path delay within O-D $(i,\,j)$, and is equal to $\Theta_{ij}[Q^*]$ according to Definition \ref{dueelasticdef}.

As an immediate consequence of  \eqref{chapVI:eqn3} and \eqref{chapVI:eqn4}, the following inequality holds for all $(h,\,Q)\in\widetilde\Lambda$.
\begin{equation}\label{chapVI:eqn5}
\sum_{p\in \mathcal{P}}\int\nolimits_{t_{0}}^{t_{f}}\Psi _{p}(t,h^*)(h_{p}(t)-h_{p}^*(t))dt
-\sum_{(i,\,j) \in \mathcal{W}}\Theta _{ij}[Q^*] (Q_{ij} -Q_{ij}^*)~\geq~0
\end{equation}
\noindent which is the desired variational inequality. 

\vspace{0.2 in}

{\bf [Sufficiency].} Assume that \eqref{chapVI:eqn5} holds for any $(h,\,Q)\in\widetilde\Lambda$. Then $h^*$ must be a solution of the fixed-demand DUE problem with fixed demand  given by $Q^*$, as in this case the second term in the left hand side of \eqref{chapVI:eqn5} vanishes and we recover the well-known VI for the fixed demand case; see Theorem 2 of \cite{Friesz1993}. By definition of fixed-demand DUE,
 for any $(i,\,j)\in\mathcal{W}$,
\begin{equation}\label{fdduedef}
\begin{array}{c}
h_{p}^{*}(t) >0,~p\in \mathcal{P}_{ij} \Longrightarrow~~ \Psi _{p}(t,\,h^*) ~=~v_{ij}(h^*) \qquad \hbox{for almost every } t \in[t_0,\,t_f]
\\
\Psi_p(t,\,h^*)~\geq~v_{ij}(h^*) \qquad\forall p\in\mathcal{P}_{ij}\qquad \hbox{for almost every } t \in[t_0,\,t_f]
\end{array}
\end{equation}
In order to show that $(h^*,\,Q^*)$ is a DUE with elastic demand using Definition \ref{dueelasticdef}, we fix an arbitrary O-D pair $(k,\,l)\in\mathcal{W}$, and distinguish between two cases: 
\begin{itemize}
\item $Q_{kl}^*>0$.  We define the following pair $(\hat h,\,\hat Q)\in\widetilde\Lambda$: 
\begin{align*}
\hat h_p(t)~=~\begin{cases}
a\,h_p^*(t)\quad &p\in\mathcal{P}_{kl}\\ 
h_p^*(t) \quad & p\in\mathcal{P}\setminus \mathcal{P}_{kl}\end{cases}\quad\forall t\in[t_0,\,t_f], \qquad \hat Q_{ij}~=~
\begin{cases} 
a\,Q^*_{kl}\quad &(i,\,j)=(k,\,l)
\\
Q^*_{ij}\quad & (i,\,j)\in \mathcal{W}\setminus (k,\,l)\end{cases}
\end{align*}
\noindent where $a>0$ is an arbitrary positive parameter.  Substituting $(h,\, Q)$ for $(\hat h,\,\hat Q)$, the left hand side of  \eqref{chapVI:eqn5} becomes
\begin{align*}
&\sum_{p\in \mathcal{P}}\int\nolimits_{t_{0}}^{t_{f}}\Psi _{p}(t,h^*)(\hat h_{p}-h_{p}^{\ast })dt
-\sum_{(i,\,j) \in \mathcal{W}}\Theta _{ij}\left[ Q^{\ast } \right] (\hat Q_{ij} -Q_{ij}^*)
\\
~=~&\sum_{p\in\mathcal{P}_{kl}}\int_{t_0}^{t_f}\Psi_p(t,\,h^*)(ah_p^*-h_p^*)\,dt+\sum_{p\in\mathcal{P}\setminus\mathcal{P}_{kl}}\int_{t_0}^{t_f}\Psi_p(t,\,h^*)(h_p^*-h_p^*)\,dt
\\
&~-~\Theta_{kl}[Q^*]\left[aQ^*_{kl}-Q^*_{kl}\right]-\sum_{(i,\,j)\in\mathcal{W}\setminus (k,\,l)}\Theta_{ij}[Q^*](Q_{ij}^*-Q^*_{ij})
\\
~=~&(a-1)\sum_{p\in\mathcal{P}_{kl}}\int_{t_0}^{t_f}\Psi_p(t,\,h^*)\,h_p^*(t)\,dt-(a-1)\Theta_{kl}[Q^*]\,Q^*_{kl}
\\
~=~&(a-1) v_{kl}(h^*) Q^*_{kl}-(a-1)\Theta_{kl}[Q^*]Q^*_{kl}
\end{align*}
We conclude from \eqref{chapVI:eqn5} that 
$$
(a-1)\big(v_{kl}(h^*)-\Theta_{kl}[Q^*]\big)Q^*_{kl}~\geq~0
$$
Since $a\in\mathbb{R}_{++}$ is arbitrary and $Q^*_{kl}>0$, there must hold $v_{kl}(h^*)=\Theta_{kl}[Q^*]$.  Thus replacing $v_{kl}(h^*)$ with $\Theta_{kl}[Q^*]$ in \eqref{fdduedef} yields \eqref{chapVI:eqn1}-\eqref{chapVI:eqn2}.

\item $Q_{kl}^*=0$. In this case, we consider $(\hat h,\,\hat Q)\in\widetilde\Lambda$ such that $h_p(t)\equiv h_p^*(t)$ for all $t\in[t_0,\,t_f]$, $p\in\mathcal{P}\setminus\mathcal{P}_{kl}$. Substituting $(h,\,Q)$ for $(\hat h,\,\hat Q)$ in \eqref{chapVI:eqn5} yields
\begin{equation}
\sum_{p\in\mathcal{P}_{kl}}\int_{t_0}^{t_f}\Psi_p(t,\,h^*)(h_p(t)-h_p^*(t))dt - \Theta_{kl}[Q^*] Q_{kl}~\geq~0
\end{equation}
\noindent which implies that 
$$
\sum_{p\in\mathcal{P}_{kl}}\int_{t_0}^{t_f}\big(\Psi_p(t,\,h^*)-\Theta_{kl}[Q^*]\big) \, h_p(t)\,dt ~\geq~0
$$
\noindent Since the vector $(h_p(t):\,p\in\mathcal{P}_{kl},\,t\in[t_0,\,t_f])$ is arbitrary, we must have that $\Psi_p(t,\,h^*)\geq \Theta_{kl}[Q^*]$ for any $p\in\mathcal{P}_{kl}$ and almost every $t\in[t_0,\,t_f]$, which is \eqref{chapVI:eqn2}. Finally, \eqref{chapVI:eqn1} is trivially true since $h_p^*(t)\equiv 0$ for all $p\in\mathcal{P}_{kl}$ and almost every $t\in[t_0,\,t_f]$.
\end{itemize}

Combining the above two cases, we finish the proof.\end{proof}

\section{Proof of Theorem \ref{existencethm}}\label{secappexistencethm}

\begin{proof} The proof is completed in several steps. \\

\noindent {\bf Step 1.} Since Theorem \ref{browderthm}  cannot be directly applied to obtain a solution of the infinite-dimensional VI, let us instead employ the technique from \cite{existence} by considering finite-dimensional approximations of $\widetilde\Lambda$. More specifically, consider, for each $n\geq 1$, the uniform partition of $[t_0,\,t_f]$ into $n$ sub-intervals $I^1,\ldots, I^n$. Define the finite-dimensional subset of $\widetilde\Lambda$:
\begin{equation}
\widetilde\Lambda_n~\doteq~\Bigg\{(h,\,Q)\in \widetilde\Lambda: 
\\
 h_p(\cdot)~\hbox{is  constant on each } I^j\quad \forall 1\leq j\leq n,\quad \forall p\in\mathcal{P}   \Bigg\}
\end{equation}
\noindent Clearly $\widetilde\Lambda_n\subset\widetilde\Lambda$, thus all the assumptions regarding $\mathcal{F}$ or $\Psi$ continue to hold on this smaller set.

It is not restrictive to assume that there is an upper bound on the elastic demand for each origin-destination pair. That is, there exists a vector $U=\big(U_{ij}\big)\in \mathbb{R}_{++}^{|\mathcal{W}|}$ such that 
$$
0~\leq~Q_{ij}~\leq~U_{ij} \qquad \forall (i,\,j)\in\mathcal{W},\qquad  \forall (h,\,Q)\in\widetilde\Lambda
$$
\noindent We claim that $\widetilde\Lambda_n$ defined as such is convex and compact in $\widetilde\Lambda$ for each $n\geq 1$. We begin with convexity. Let $\widehat X=(\widehat h,\,\widehat Q)$ and $\overline{X}=(\overline h,\,\overline{Q})$ be any two points in $\widetilde\Lambda_n$. Given any $\alpha\in(0,\,1)$, we have that
$$
\sum_{p\in\mathcal{P}_{ij}}\int_{t_0}^{t_f}\Big(\alpha\,\widehat h_p(t)\,dt+(1-\alpha) \, \overline h_p(t)\Big)\,dt~=~\alpha\, \widehat Q_{ij}+(1-\alpha)\, \overline Q_{ij}\qquad \forall (i,\,j)\in\mathcal{W}
$$
Moreover, $ \alpha\, \widetilde h_p(\cdot)+(1-\alpha)\,\overline h_p(\cdot)$ clearly remains constant on each sub-interval $I^j,\,j=1,\ldots, n$, $\forall p\in\mathcal{P}$. We thus conclude that $\alpha \widehat X+(1-\alpha)\overline X\in \widetilde \Lambda_n$.

Next, let us investigate compactness. Let us fix $n\geq 1$. In light of Proposition  \ref{propsition3}, it suffices to establish sequential compactness for $\widetilde\Lambda_n$. We consider an arbitrary infinite sequence $\big\{X^k\big\}_{k\geq 1}\subset \widetilde\Lambda_n$ where $X^k=\big(h^k,\,Q^k\big)$. For each $k\geq 1$ and $p\in\mathcal{P}$, let $\mu_p^k=(\mu_{p,j}^k:\, j=1,\ldots,n)\in \mathbb{R}_+^{n}$ be such that
$$
h^k_p(t)~\equiv~\mu_{p,j}^k \qquad \forall t\in I^j,\quad \forall j~=~1,\,\ldots,\,n
$$
\noindent We then define $\mu^k\in\mathbb{R}_+^{n\times |\mathcal{P}|}$ to be the concatenation of all vectors $\mu_p^k,\,p\in\mathcal{P}$. We also notice that the sequence $\{\mu^k\}_{k\geq 1}$ are uniformly bounded by the constant 
$$
\max_{(i,\,j)\in\mathcal{W}} U_{ij} \cdot { n\over t_f-t_0 }
$$
\noindent Thus by the Bolzano-Weierstrass theorem, there exists a convergent subsequence $\big\{\mu^{k'}\big\}_{k'\geq 1}$. It is immediately verifiable that the corresponding subsequence $\{h^{k'}\}_{k'\geq 1}$ converge uniformly on $[t_0,\,t_f]$ and also in the $L^2$ norm.  Moreover, by virtue of the uniform bounds $U_{ij},\,(i,\,j)\in\mathcal{W}$, there exists a further subsequence $\{Q^{k''}\}_{k''\geq 1}$ that converge according to the Bolzano-Weierstrass theorem. Thus, the subsequence $\big\{X^{k''}\big\}_{k''\geq 1}$ converges with respect to the norm $\|\cdot\|_E$. \\

\noindent {\bf Step 2.}  We fix $n\geq 1$ and consider $\widetilde\Lambda_n$. Due to the uniform upper bound on all the elastic demands, the departure rate vectors that are piecewise constant must be pointwise and uniformly bounded. Thus, {\bf (A3)} immediately implies that $\Psi$ is continuous when restricted to the set of piecewise constant path departure vectors. Combining this with the fact that $\Theta$ is continuous, we conclude that $\mathcal{F}$ defined in \eqref{calfdef} is continuous on $\widetilde\Lambda_n$. Given this and the convexity and compactness of $\widetilde\Lambda_n$ established earlier, Theorem \ref{browderthm} asserts that  there exists some $X^{n,*}=(h^{n,*},\,Q^{n,*})\in\widetilde\Lambda_n$ such that 
\begin{equation}\label{chapExistence:eqn2}
\left<\mathcal{F}\left(X^{n, *}\right),\,X^n-X^{n,*}\right>_E~\geq~0\qquad\forall X^n\in\widetilde\Lambda_n
\end{equation}

\noindent As a slight reformulation of  \eqref{chapExistence:eqn2}, we have
\begin{equation}\label{eqn1}
\sum_{p\in\mathcal{P}}\int_{t_0}^{t_f}\Psi_p\big(t,\,h^{n,*}\big) h_p^{n,*}(t)\,dt-\sum_{(i,\,j)\in\mathcal{W}}\Theta_{ij}[Q^{n,*}]Q_{ij}^{n,*} ~\leq~\sum_{p\in\mathcal{P}}\int_{t_0}^{t_f}\Psi_p\big(t,\,h^{n,*}\big) h_p^{n}(t)\,dt-\sum_{(i,\,j)\in\mathcal{W}}\Theta_{ij}[Q^{n,*}]Q_{ij}^n
\end{equation}
\noindent for all $\big(h^n,\,Q^n\big)\in\widetilde\Lambda_n$. In particular, for all $(h^n,\,Q^n)\in\widetilde\Lambda_n$ such that 
$$
\sum_{p\in\mathcal{P}_{ij}}\int_{t_0}^{t_f}h_p^n(t)\,dt~=~Q^{n,*}_{ij}\qquad\forall (i,\,j)\in\mathcal{W},
$$
\noindent inequality \eqref{eqn1} becomes 
\begin{equation}\label{eqn2}
\sum_{p\in\mathcal{P}}\int_{t_0}^{t_f}\Psi_p\big(t,\,h^{n,*}\big) h_p^{n,*}(t)\,dt~\leq~\sum_{p\in\mathcal{P}}\int_{t_0}^{t_f}\Psi_p\big(t,\,h^{n,*}\big) h_p^{n}(t)\,dt
\end{equation}
Recall that $h^{n,*}(\cdot)$ is piecewise constant, thus \eqref{eqn2} implies the following: 
\begin{equation}\label{eqn3}
h^{n,*}_p(t)~>~0, \quad t\in I^k ~\Longrightarrow~ \int_{I^k}\Psi_p(t,\,h^{n,*}) \,dt ~=~\min_{1\leq j\leq n} \int_{I^j}\Psi_p(t,\,h^{n,*}) \,dt
\end{equation}
\noindent for all $p\in\mathcal{P}_{ij}$, $(i,\,j)\in\mathcal{W}$. \\

\noindent {\bf Step 3.} We assert that there must exist a constant $\mathcal{M}>0$ such that 
\begin{equation}\label{eqn4}
h_p^{n,*}(t)~\leq~\mathcal{M}\qquad\forall t\in[t_0,\,t_f],\quad\forall p\in\mathcal{P},\quad\forall n\geq 1
\end{equation}
Indeed, given {\bf (A2)}, we are prompted to define the following constant
$$
M^{max}~\doteq~\max_{a\in\mathcal{A}} M_a~<~+\infty
$$
where $\mathcal{A}$ is the set of links in the network. Recalling the constant $\Delta$ from {\bf (A1)}, we set $\mathcal{M}$ such that
$$
\mathcal{M}~>~{3 M^{max}\over \Delta+1}
$$
\noindent We claim that \eqref{eqn4} holds for such $\mathcal{M}$.  Otherwise, if \eqref{eqn4} fails, there must exist some $m\geq 1,\, q\in\mathcal{P}$ and $1\leq j\leq m$ such that
$$
h_q^{m,*}(t)~\equiv~\lambda~>~\mathcal{M}\qquad t\in I^j
$$
Without losing generality, we assume that $j> 1$ and consider the interval $I^{j-1}$. By possibly modifying the value of the function $\Psi_q(\cdot,\,h^{m,*})$ at one point without changing the measure-theoretic nature of the problem, we obtain $t^*\in I^{j-1}$ such that
$$
\Psi_q\big(t^*,\,h^{m,*}\big)~=~\sup_{t\in I^{j-1}} \Psi_q\big(t,\,h^{m,*}\big)
$$
We denote by $\tau_q(t,\, h^{m,*})$ the time of arrival at destination of driver who departs at time $t$ along path $q$. According to the first-in-first-out (FIFO) principle, we deduce that $\forall t\in I^j$,
$$
(t-t^j)\lambda~\leq~\int_{t^*}^t h^{m,*}_q(t)\,dt~\leq~M^{max}\big(\tau_q(t,\,h^{m,*})-\tau_q(t^*,\,h^{m,*})\big)
$$
where $t^j$ is the left boundary of the interval $I^j$. We then have the following estimation:
\begin{align}
&\Psi_q(t,\,h^{m,*})-\Psi_q(t^*,\,h^{m,*})
\\
~=~&D_q(t,\,h^{m,*})+f\big(\tau_q(t,\,h^{m,*})- T_A\big)-D_q(t^*,\,h^{m,*})-f\big(\tau_q(t^*,\,h^{m,*})- T_A\big)   \nonumber
\\
~\geq~&\tau_q(t,\,h^{m,*})-\tau_q(t^*,\,h^{m,*})-(t-t^*)+\Delta\big(\tau_q(t,\,h^{m,*})-\tau_q(t^*,\,h^{m,*})\big)    \nonumber
\\
~=~&(\Delta+1)\big(\tau_q(t,\,h^{m,*})-\tau_q(t^*,\,h^{m,*})\big)-(t-t^*)    \nonumber
\\
\label{eqn5}
~\geq~&(\Delta+1){\lambda\over M^{max}}(t-t^j) -(t-t^*)\qquad \qquad \forall t\in I^j
\end{align}
\noindent Integrating \eqref{eqn5} with respect to $t$ over interval $I^j$ shows the following:
\begin{equation}\label{eqn6}
\int_{I^j}\Psi_q(t,\,h^{m,*})\,dt - (t^{j+1}-t^j)\Psi_q(t^*,\,h^{m,*})~\geq~{(t^{j+1}-t^j)^2\over 2}\cdot {(\Delta+1)\lambda \over M^{max}}+(t^{j+1}-t^j)\cdot\left(t^*-{t^j+t^{j+1}\over 2}\right)
\end{equation}
\noindent where $t^j,\,t^{j+1}$ are respectively the left and right boundary of $I^j$. Since $t^*\in I^{j-1}$, we have that 
$$
t^*-{t^j+t^{j+1}\over 2}~\geq~-{3\over 2}\big(t^{j+1}-t^j\big)
$$
With this observation, \eqref{eqn6} becomes
\begin{align*}
\int_{I^j}\Psi_q(t,\,h^{m,*})\,dt - (t^{j+1}-t^j)\Psi_q(t^*,\,h^{m,*})~\geq~&{(t^{j+1}-t^j)^2\over 2}\cdot {(\Delta+1)\lambda \over M^{max}}-{3\over 2}\big(t^{j+1}-t^j\big)^2
\\
~=~&{(t^{j+1}-t^j)^2\over 2}\left({(\Delta+1)\lambda\over M^{max}}-3\right)~>~0
\end{align*}
This implies
$$
\int_{I^j}\Psi_q(t,\,h^{m,*})\,dt~>~\int_{I^{j-1}}\Psi(t,\,h^{m,*})\,dt
$$
which yields contradiction to \eqref{eqn3}. This substantiates our claim. \\

\noindent {\bf Step 4.} With the point-wise uniform upper bound $\mathcal{M}$ on the path departure rate vectors $h^{n,*}$ for all $n\geq 1$, by taking a subsequence if necessary, one can assume the weak convergence $h^{n,*}\to h^*\in\big(L_+^2[t_0,\,t_f]\big)^{|\mathcal{P}|}$. Let 
$$
Q_{ij}^*~\doteq~\sum_{p\in\mathcal{P}_{ij}}\int_{t_0}^{t_f}h_p^*(t)\,dt \quad\forall (i,\,j)\in\mathcal{W},\qquad Q^*\doteq\big(Q_{ij}^*:\,(i,\,j)\in\mathcal{W}\big)
$$
\noindent We have that for any $(i,\,j)\in\mathcal{W}$,
$$
Q_{ij}^*~=~\sum_{p\in\mathcal{P}_{ij}}\int_{t_0}^{t_f}h^*_p(t)\,dt~=~\lim_{n \to \infty} \sum_{p\in\mathcal{P}_{ij}}\int_{t_0}^{t_f} h_p^{n,*}(t)\,dt~=~\lim_{n \to\infty} Q_{ij}^{n,*}
$$
\noindent Therefore, we have that $X^{n,*}\doteq (h^{n,*},\,Q^{n,*}) \to X^*\doteq (h^*,\,Q^*)$ weakly. According to {\bf (A3)}, we have that $\Psi_p(t,\,h^{n,*})\to \Psi_p(t,\,h^*)$ uniformly for all $t\in[t_0,\,t_f]$ and $p\in\mathcal{P}$, and thus $\mathcal{F}(X^{n,*}) \to \mathcal{F}(X^*)$ strongly.

Finally, we show that
$$
\left<\mathcal{F}(X^*),\,X-X^*\right>_E~\geq~0 \qquad\forall X\in\widetilde\Lambda
$$
\noindent Indeed, we consider any $X\in\widetilde\Lambda$, and choose a sequence $X^n\in\widetilde\Lambda_n$ such that $\|X^n-X\|_E \to 0$ as $n\to\infty$. For every $n$, by \eqref{chapExistence:eqn2} one has
\begin{equation}\label{Xnineq}
\left<\mathcal{F}\left(X^{n, *}\right),\,X^n-X^{n,*}\right>_E~\geq~0
\end{equation}
\noindent This leads to the following:
\begin{align}
&\left<\mathcal{F}(X^*),\, X\right>_E  - \left<\mathcal{F}(X^*),\, X^*\right>_E \nonumber
\\
~=~&\Big[ \left<\mathcal{F}(X^*),\,X\right>_E - \left<\mathcal{F}(X^{n, *}),\,X^n\right>_E \Big]  +   \left<\mathcal{F}(X^{n, *}),\,X^n-X^{n,*}\right>_E  \nonumber
\\
&+   \left<\mathcal{F}(X^{n,*})-\mathcal{F}(X^*),\, X^{n,*}\right>_E + \left<\mathcal{F}(X^*),\, X^{n,*}-X^*\right>_E  \nonumber
\\
\label{XBnineq}
~\geq~& \Big[ \left<\mathcal{F}(X^*),\,X\right>_E - \left<\mathcal{F}(X^{n, *}),\,X^n\right>_E \Big] + \left<\mathcal{F}(X^{n,*})-\mathcal{F}(X^*),\, X^{n,*}\right>_E + \left<\mathcal{F}(X^*),\, X^{n,*}-X^*\right>_E  \quad\forall n
\end{align}
\noindent By strong convergence $\mathcal{F}(X^{n,*})\to\mathcal{F}(X^*)$ and $X^n\to X$, the term in the square brackets in \eqref{XBnineq} converges to zero. The second term of \eqref{XBnineq} converges to zero due to 
$$
 \left<\mathcal{F}(X^{n,*})-\mathcal{F}(X^*),\, X^{n,*}\right>_E ~\leq~\left\|\mathcal{F}(X^{n,*})-\mathcal{F}(X^*)\right\|_E\cdot \left\|X^{n,*}\right\|_E
$$ 
\noindent and the fact that the first factor converges to zero by strong convergence, and the second factor is uniformly bounded for all $n$. Finally, the third term in \eqref{XBnineq} converges to zero by weak convergence.

Thus, taking the limit $n\to \infty$ in \eqref{XBnineq} yields
$$
\left<\mathcal{F}(X^*),\, X\right>_E  - \left<\mathcal{F}(X^*),\, X^*\right>_E~\geq~0
$$
\noindent Since $X$ is arbitrary, $X^*$ is a solution of the variational inequality.

 \end{proof}

\section{Proof of Theorem \ref{thmexplicitproj}}\label{secappthmexplicitproj}
\begin{proof} 
According to the definition of the minimum-norm projection, at the $k$-th iteration, we need to solve the following minimization problem:
\begin{equation}\label{chapAlg:eqn20}
X^{k+1}~=~\underset{X\in\widetilde\Lambda_1}{\hbox{argmin}}\left\{{1\over 2}\left\|X^k-\alpha\mathcal{F}(X^k)-X\right\|^2\right\},
\end{equation}
\noindent which is recognized as a linear-quadratic optimal control problem:
\begin{multline}\label{secappeqn1}
X^{k+1}=\Big(h^{k+1},\,y^{k+1}(t_f)\Big)~=~\underset{X=(h,\,y(t_f))}{\hbox{argmin}}\int_{t_0}^{t_f}{1\over 2}\sum_{(i,j)\in\mathcal{W}}\sum_{p\in\mathcal{P}_{ij}}\left[h^k_p(t)-\alpha \Psi_p(t,\,h^k)-h_p(t)\right]^2dt\\
+{1\over 2}\sum_{(i,j)\in\mathcal{W}}\left(y_{ij}^k(t_f)-\alpha\Theta_{ij}^-\big[y^k(t_f)\big] - y_{ij}(t_f)\right)^2
\end{multline}
subject to
\begin{align}
\label{chapAlg:eqn21}
{dy_{ij}(t)\over dt} ~=~&\sum_{p\in \mathcal{P}_{ij}}h_{p}\left( t\right) \qquad \forall \left( i,j\right) \in \mathcal{W} \\
\label{chapAlg:eqn22}
y_{ij}(t_{0}) ~=~& 0\qquad \qquad \forall \left( i,j\right) \in \mathcal{W}
\\
\label{chapAlg:eqn23}
h ~\geq~& 0 
\end{align}

\noindent The Hamiltonian for the above optimal control problem is 
$$
H~=~{1\over 2}\sum_{(i,j)\in\mathcal{W}}\sum_{p\in\mathcal{P}_{ij}}\left[h_p^k(t)-\alpha\Psi_p(t,\,h^k)-h_p(t)\right]^2 +\sum_{(i,j)\in\mathcal{W}}\lambda_{ij}(t)\sum_{p\in\mathcal{P}_{ij}}h_p(t),
$$
\noindent for which the adjoint equations are
\begin{equation}\label{chapAlg:constantlambda}
{d\lambda_{ij}(t)\over dt}~=~-{\partial H\over \partial y_{ij}}~=~0\qquad\forall (i,\,j)\in\mathcal{W},\quad \forall t\in[t_0,\,t_f];
\end{equation}
\noindent and the transversality conditions read: for all $(i,\,j)\in\mathcal{W}$,
\begin{equation}\label{etra}
\lambda_{ij}(t_f)~=~{1\over 2}{\partial \sum_{(k,l)\in\mathcal{W}}\left(y_{kl}^k(t_f)-\alpha\Theta_{kl}^-\big[y^k(t_f)\big]-y_{kl}(t_f)\right)^2\over \partial y_{ij}(t_f)}~=~-y_{ij}^k(t_f)+\alpha\Theta_{ij}^-\big[y^k(t_f)\big]+y_{ij}(t_f)
\end{equation}
\noindent According to the minimum principle, we enforce the following minimization problem
$$
\min_h H \quad \hbox{s.t.} \quad -h~\leq~0,
$$
\noindent for which the Kuhn-Tucker conditions are (we denote the K-T point by $h^{k+1}$)
\begin{align}
\label{kt1}
-\left[h_p^k(t)-\alpha \Psi_p(t,\,h^k)-h^{k+1}_p(t)\right]+\lambda_{ij}(t)&~=~\rho_p(t) \qquad\forall (i,\,j)\in\mathcal{W},\quad p\in\mathcal{P}_{ij},\quad\forall t\in[t_0,\,t_F]\\
\label{kt2}
\rho_p(t)\cdot h_p^{k+1}(t)&~=~0\qquad\forall (i,j)\in\mathcal{W},\quad p\in\mathcal{P}_{ij},\quad\forall t\in[t_0,\,t_F]\\
\label{kt3}
\rho_p(t)&~\geq~0\qquad \forall (i,j)\in\mathcal{W},\quad p\in\mathcal{P}_{ij},\quad\forall t\in[t_0,\,t_F]
\end{align}

\noindent Thus, the optimality conditions for system \eqref{secappeqn1}-\eqref{chapAlg:eqn23} are:
\begin{equation}\label{chapAlg:eqn235}
\lambda_{ij}(t)~\equiv~-y_{ij}^k(t_f)+\alpha\Theta^-_{ij}[y^k(t_f)]+y^{k+1}_{ij}(t_f)\qquad \forall (i,\,j)\in\mathcal{W}
\end{equation}
\begin{align}
0 \leq h^{k+1}_p(t)=&~\hbox{arg}\left\{{\partial H\over \partial h_p}=0\right\}=\hbox{arg}\left\{-\left(h_p^k(t)-\alpha\Psi_p(t,\,h^k)-h_p(t)\right)+\lambda_{ij}=0\right\}\nonumber
\\
\label{chapAlg:eqn24}
=&~\hbox{arg}\left\{-\left(h_p^k(t)-\alpha\Psi_p(t,\,h^k)-h_p(t)\right)-y_{ij}^k(t_f)+\alpha\Theta_{ij}^-[y^k(t_f)]+y_{ij}^{k+1}(t_f)=0\right\}
\end{align}
\noindent In other words, given $(i,\,j)\in\mathcal{W}$ and $p\in\mathcal{P}_{ij}$, $h_p^{k+1}(\cdot)$ is determined as 
\begin{align}
h_p^{k+1}(t)~=~& \left[h_p^k(t) -\alpha\Psi_p(t,\,h^k)+y_{ij}^k(t_f)-\alpha\Theta_{ij}^-[y^k(t_f)]-y_{ij}^{k+1}(t_f)\right]_+  \nonumber
\\
\label{chapAlg:eqn25}
~=~&\left[h_p^k(t)-\alpha\Psi_p(t,\,h^k)+Q_{ij}^k+\alpha\Theta_{ij}[Q^k]-Q_{ij}^{k+1}\right]_+
\end{align}
\noindent where 
$$
Q^k_{ij}~\equiv~y_{ij}^k(t_f)~=~\sum_{p\in\mathcal{P}_{ij}}\int_{t_0}^{t_f}h^k_p(t)\,dt\qquad \forall (i,\,j)\in\mathcal{W},\quad\forall k,
$$
\noindent and $[x]_+\doteq \max\{0,\,x\}$ for $x\in\mathbb{R}$. Notice that for all $(i,\,j)\in\mathcal{W}$,  $Q_{ij}^{k+1}$ must satisfy
\begin{equation}\label{chapAlg:eqn26}
Q_{ij}^{k+1}~=~\sum_{p\in\mathcal{P}_{ij}}\int_{t_0}^{t_f}h_p^{k+1}(t)dt ~=~\sum_{p\in\mathcal{P}_{ij}}\int_{t_0}^{t_f}\Big[h_p^k(t)-\alpha\Psi_p(t,\,h^k)+Q_{ij}^k+\alpha\Theta_{ij}\big[Q^k\big]-Q_{ij}^{k+1}\Big]_+ dt
\end{equation}

Finally, we show the existence and uniqueness of $Q_{ij}^{k+1}\in\mathbb{R}_+$ that satisfies \eqref{chapAlg:eqn26}. We slightly rewrite \eqref{chapAlg:eqn26} as 
\begin{equation}\label{chapAlg:eqn265}
\sum_{p\in\mathcal{P}_{ij}}\int_{t_0}^{t_f}\Big[h_p^k(t)-\alpha\Psi_p(t,\,h^k)+Q_{ij}^k+\alpha\Theta_{ij}\big[Q^k\big]-Q_{ij}^{k+1}\Big]_+ dt-Q_{ij}^{k+1}~=~0
\end{equation}

\noindent Two cases may arise:
\begin{itemize}
\item[(i)] $h_p^k(t)-\alpha\Psi_p(t,\,h^k)+Q_{ij}^k+\alpha\Theta_{ij}[Q^k]\leq 0$ for almost every $t \in[t_0,\,t_f]$ and all $p\in\mathcal{P}_{ij}$. Clearly $Q_{ij}^{k+1}=0$ is the only solution to \eqref{chapAlg:eqn265} in this case.

\item[(ii)] $h_p^k(t)-\alpha\Psi_p(t,\,h^k)+Q_{ij}^k+\alpha\Theta_{ij}[Q^k] > 0$ for some $p\in\mathcal{P}_{ij}$ and for $t\in\mathcal{B}\subset [t_0,\,t_f]$ where $\mathcal{B}$ is a set with  positive measure. We call the left hand side of \eqref{chapAlg:eqn265} $f(Q_{ij}^{k+1})$, which is a continuous function of $Q_{ij}^{k+1}$. According to the hypothesis, the following hold
$$
f(0)~>~0,\qquad\qquad f(Q_{ij}^{k+1})~<~0\quad\hbox{if}\quad Q_{ij}^{k+1}~\hbox{is very large}
$$
 Therefore, by the Intermediate Value Theorem, there must exist at least one value of  $Q_{ij}^{k+1}$ such that $f(Q_{ij}^{k+1})$ vanishes.   The uniqueness of such a solution follows by observing that $f(\cdot)$, as a function of $Q_{ij}^{k+1}$, is strictly decreasing. 
 \end{itemize}
 
 Therefore, $X^{k+1}=(h^{k+1},\,Q^{k+1})$ given by \eqref{chapAlg:eqn25} and \eqref{chapAlg:eqn26}  is unique.
\end{proof}

\section{Proof of Theorem \ref{convthmmswm}}\label{secappmswmthm}

\begin{proof}
Using the non-expansiveness property of the projection operator, we have that 
\begin{align}
&\left\|P_{\widetilde\Lambda}\big[X^{k+1}-\alpha\mathcal{F}(X^{k+1})\big] - P_{\widetilde\Lambda}\big[X^{k}-\alpha\mathcal{F}(X^{k})\big] \right\|_E^2 \nonumber
\\
~\leq~& \left\|X^{k+1}-\alpha\mathcal{F}(X^{k+1}) - X^k +\alpha\mathcal{F}(X^k)\right\|_E^2 \nonumber
\\
\label{MSWMproof1}
~=~&\left\|X^{k+1}-X^k\right\|_E^2 -2\alpha \left<\mathcal{F}(X^{k+1})-\mathcal{F}(X^k),\, X^{k+1}-X^k\right>_E+\alpha^2\left\|\mathcal{F}(X^{k+1})-\mathcal{F}(X^k)\right\|_E^2 
\end{align}
\noindent Recalling that $X^{k+1}=\big(h^{k+1},\,Q^{k+1}\big)$ and $X^{k}=\big(h^{k},\,Q^{k}\big)$, we deduce from the hypotheses that
\begin{align}
&\left<\mathcal{F}(X^{k+1})-\mathcal{F}(X^k),\, X^{k+1}-X^k\right>_E  \nonumber
\\
~=~&\left<\Psi(h^{k+1})-\Psi(h^k),\,h^{k+1}-h^k\right>+\left<\Theta^-[Q^{k+1}]-\Theta^-[Q^k],\,Q^{k+1}-Q^k\right>  \nonumber
\\
~=~&\sum_{p\in\mathcal{P}^{sm}}\left<\Psi_p(h^{k+1})-\Psi_p(h^k),\,h_p^{k+1}-h_p^k\right>+\sum_{q\in\mathcal{P}\setminus\mathcal{P}^{sm}}\left<\Psi_q(h^{k+1})-\Psi_q(h^k),\,h_q^{k+1}-h_q^k\right>  \nonumber
\\
~+~& \left<\Theta^-[Q^{k+1}]-\Theta^-[Q^k],\,Q^{k+1}-Q^k\right>  \nonumber
\\
~\geq~&K^{sm}\sum_{p\in\mathcal{P}^{sm}}\left\|h_p^{k+1}-h_p^k\right\|^2-K^{wm}\sum_{q\in\mathcal{P}\setminus \mathcal{P}^{sm}}\left\|h_q^{k+1}-h_q^k\right\|^2+K_2\left\|Q^{k+1}-Q^k\right\|^2  \nonumber
\\
~\geq~&K^{sm}\sum_{p\in\mathcal{P}^{sm}}\left\|h_p^{k+1}-h_p^k\right\|^2-K^{wm}M\sum_{p\in\mathcal{P}^{sm}}\left\|h_p^{k+1}-h_p^k\right\|^2 + K_2\left\|Q^{k+1}-Q^k\right\|^2  \nonumber
\\
~=~&(K^{sm}-K^{wm}M)\sum_{p\in\mathcal{P}^{sm}}\left\|h_p^{k+1}-h_p^k\right\|^2+K_2\left\|Q^{k+1}-Q^k\right\|^2 \nonumber
\\
\label{MSWMproof2}
~\geq~&{K^{sm}-K^{wm}M\over M+1}\left\|h^{k+1}-h^k\right\|^2_{L^2}+K_2\left\|Q^{k+1}-Q^k\right\|^2
\end{align}

\noindent Combining \eqref{MSWMproof1} and \eqref{MSWMproof2} yields
\begin{align*}
&\left\|P_{\widetilde\Lambda}\big[X^{k+1}-\alpha\mathcal{F}(X^{k+1})\big] - P_{\widetilde\Lambda}\big[X^{k}-\alpha\mathcal{F}(X^{k})\big] \right\|_E^2
\\
~\leq~& \left\|X^{k+1}-X^k\right\|_E^2  -  2\alpha\left({K^{sm}-K^{wm}M\over M+1}\left\|h^{k+1}-h^k\right\|^2_{L^2}+K_2\left\|Q^{k+1}-Q^k\right\|^2 \right)
\\
~+~&\alpha^2\Bigg( \left\|\Psi(h^{k+1})-\Psi(h^{k})\right\|_{L^2}^2 + \left\|\Theta^-[Q^{k+1}]-\Theta^-[Q^k]\right\|^2\Bigg)
\\
~\leq~& \left\|X^{k+1}-X^k\right\|_E^2-  2\alpha\left({K^{sm}-K^{wm}M\over M+1}\left\|h^{k+1}-h^k\right\|^2_{L^2}+K_2\left\|Q^{k+1}-Q^k\right\|^2 \right)
\\
~+~&\alpha^2\Bigg( L_1^2\left\|h^{k+1}-h^k\right\|_{L^2}^2+L_2^2\left\|Q^{k+1}-Q^k\right\|^2\Bigg)
\\
~\leq~& \left\|X^{k+1}-X^k\right\|_E^2-2\alpha\min\left\{{K^{sm}-K^{wm}M\over M+1},\, K_2\right\}\left\|X^{k+1}-X^k\right\|_E^2 +\alpha^2\max\left\{L_1^2,\,L_2^2\right\}\left\|X^{k+1}-X^k\right\|_E^2
\\
~=~&\left(1-2\alpha\min\left\{{K^{sm}-K^{wm}M\over M+1},\, K_2\right\}+\alpha^2\max\left\{L_1^2,\,L_2^2\right\}\right)\cdot \left\|X^{k+1}-X^k\right\|_E^2
\end{align*}
Thus, by setting 
$$
1-2\alpha\min\left\{{K^{sm}-K^{wm}M\over M+1},\, K_2\right\}+\alpha^2\max\left\{L_1^2,\,L_2^2\right\}~<~1
$$
\noindent that is,
\begin{equation}\label{MSWMproof3}
\alpha~<~{2\min\left\{{K^{sm}-K^{wm}M\over M+1},\, K_2\right\}\over \max\left\{L_1^2,\,L_2^2\right\}}
\end{equation}
\noindent we obtain the contracting property of the map $X^k\mapsto P_{\widetilde\Lambda}\big[X^k-\alpha\mathcal{F}(X^k)\big]$. The convergence follows from the contracting mapping theorem \citep{Rudin}.
\end{proof}

\section{Proof of Lemma \ref{ppmconvthm}}\label{secappppm}

\begin{proof}
Let $X^d\in Y^d$. We have that for any $k\geq 0$,
\begin{equation}\label{Liueqn2}
\left\|X^{k}-X^d\right\|_E^2~=~\left\|X^{k+1}-X^k + X^d -X^{k+1}\right\|_E^2~=~\left\|X^{k+1}-X^{k}\right\|_E^2+ \left\|X^{d}-X^{k+1}\right\|_E^2 + 2\left<X^{k+1}-X^k~,~X^{d}-X^{k+1}\right>_E
\end{equation}
\noindent Taking $X=X^d$ in \eqref{ppmvi}, and combining this with \eqref{minty} yields
\begin{equation}\label{Liueqn3}
0~\geq~\left<\mathcal{F}(X^{k+1}),\,X^d-X^{k+1}\right>_E~\geq~ - a\left<X^{k+1}-X^k,\, X^d-X^{k+1}\right>_E
\end{equation}
\eqref{Liueqn2} and \eqref{Liueqn3} together yield that for all $k\geq 0$,
\begin{equation}\label{Liueqn4}
\left\|X^{k+1}-X^k\right\|^2_E~\leq~\left\|X^{k}-X^d\right\|_E^2 -\left\|X^{k+1}-X^{d}\right\|_E^2 
\end{equation}
\noindent Summing up \eqref{Liueqn4} over different values of $k$ yields
\begin{equation}\label{Liueqn5}
(k+1) \min_{0\leq i\leq k} \left\|X^{i+1}-X^i\right\|_E^2~\leq~\sum_{i=0}^k\left\|X^{i+1}-X^i\right\|_E^2~\leq~\left\|X^0-X^d\right\|_E^2- \left\|X^{k+1}-X^d\right\|_E^2
\end{equation}

\noindent For each $k\geq 0$, we introduce the notation $\mu(k)\doteq \underset{0\leq i\leq k}{\text{argmin}}\,\left\|X^{i+1} -X^{i}\right\|_E^2$. Then \eqref{Liueqn5} implies
$$
(k+1)\left\|X^{\mu(k)+1} - X^{\mu(k)}\right\|_E^2~\leq~\left\|X^0 - X^d\right\|_E^2 ~\leq~D^2
$$
\noindent and thus
\begin{equation}\label{Liueqn6}
\left\|X^{\mu(k)+1} - X^{\mu(k)}\right\|_E^2~\leq~{1\over k+1}D^2
\end{equation}
\noindent where $D$ denotes the diameter of the feasible set $\widetilde\Lambda$.  By invoking \eqref{ppmvi}, we have 
$$
\left<\mathcal{F}(X^{\mu(k)+1}+a (X^{\mu(k)+1}- X^{\mu(k)})~,~X-X^{\mu(k)+1}\right>_E~\geq~0\qquad\forall X\in\widetilde\Lambda,
$$
\noindent which immediately leads to 
\begin{align}
\left<\mathcal{F}(X^{\mu(k)+1})~,~X-X^{\mu(k)+1}\right>_E~\geq~& -a\left<X^{\mu(k)+1}-X^{\mu(k)}~,~X-X^{\mu(k)+1}\right>_E   \nonumber
\\
~\geq~& -a \left\|X^{\mu(k)+1}-X^{\mu(k)}\right\|_E\cdot \left\|X-X^{\mu(k)+1}\right\|_E \nonumber
\\
\label{Liueqn7}
~\geq~ & -a \left\|X^{\mu(k)+1}-X^{\mu(k)}\right\|_E \cdot D 
\\
~\geq~& -a \sqrt{{1\over k+1} D^2}\cdot D \qquad\forall X\in\widetilde\Lambda \nonumber
\end{align}
\end{proof}

\bibliographystyle{model2-names}
\bibliography{<your-bib-database>}

\begin{thebibliography}{00}





\bibitem[Allevi et al., 2006]{Allevi} Allevi, E., Gnudi, A., Konnov, I.V., 2006. The proximal point method for nonmonotone variational inequalities. Mathematical Methods of Operations Research 63, 553-565. 



\bibitem[Arnott et al., 1993]{ADL} Arnott, R., dePalma, A., Lindsey, R., 1993. A structural model of peak-period congestion: A traffic bottleneck with elastic demand. The American Economic Review 83 (1), 161-179.




\bibitem[Bressan and Han, 2011]{BH} Bressan, A., Han, K., 2011. Optima and equilibria for a model of traffic flow. SIAM Journal on Mathematical Analysis 43 (5), 2384-2417.


\bibitem[Bressan and Han, 2013]{BH1} Bressan, A., Han, K., 2013. Existence of optima and equilibria for traffic flow on networks. Networks and Heterogeneous Media, 8 (3), 627-648.


\bibitem[Browder, 1968]{Browder} Browder, F.E., 1968. The fixed point theory of
multi-valued mappings in topological vector spaces. Mathematische Annalen 177,
283-301.






\bibitem[Daganzo, 1994]{CTM1} 
Daganzo, C.F., 1994. The cell transmission model: A simple dynamic representation of highway traffic. Transportation Research Part B 28 (4), 269-287.


\bibitem[Daganzo, 1995]{CTM2} Daganzo, C.F., 1995. The cell transmission model, part II: network traffic. Transportation  Research Part B 29 (2), 79-93. 


\bibitem[Facchinei and Pang, 2003]{FP} Facchinei, F., Pang, J.S., 2003. Finite-Dimensional Variational Inequalities and Complementarity Problems. Springer.





\bibitem[Friesz, 2010]{DODG} Friesz, T. L., 2010. Dynamic Optimization and Differential
Games. Springer.

\bibitem[Friesz et al., 1993]{Friesz1993}
Friesz, T.L.,  Bernstein, D.,  Smith, T.,  Tobin, R., Wie, B.,  1993. A variational inequality formulation of the dynamic network user equilibrium problem. Operations Research 41  (1), 80-91.


\bibitem[Friesz et al., 2001]{FBST} Friesz, T.L., Bernstein, D., Suo, Z., Tobin, R., 2001. Dynamic network user equilibrium with state-dependent time lags. Networks and Spatial Economics 1 (3/
4), 319Ð347.


\bibitem[Friesz et al., 2013a]{FHLY} Friesz, T.L., Han, K., Liu, H., Yao, T., 2013a. Dynamic congestion and tolls with mobile source emission. Procedia - Social and Behavioral Sciences 80, 818-836. 


\bibitem[Friesz et al., 2013b]{FHNMY} Friesz, T.L., Han, K., Neto, P.A., Meimand, A., Yao, T., 2013b. Dynamic user equilibrium based on a hydrodynamic model. Transportation Research Part B 47, 102-126.


\bibitem[Friesz et al., 2011]{FKKR} Friesz, T. L., Kim, T., Kwon, C., Rigdon, M.A., 2011. Approximate dynamic network loading and dual time scale dynamic user equilibrium. Transportation Research Part B 45 (1), 176-207.



\bibitem[Friesz and Meimand, 2014]{FM2014} Friesz, T. L., Meimand, A, 2014. Dynamic user equilibria with elastic demand. Transportmetrica A: Transport Science 10 (7), 661-668.


\bibitem[Friesz and Mookherjee, 2006]{FM2006} Friesz, T.L., Mookherjee, R., 2006. Solving the dynamic network user equilibrium problem with state dependent time shifts. Transportation Research Part B 40 (3), 207-229.


 




\bibitem[Han and Lo, 2002]{HL2002} Han, D., Lo, H.K., 2002. Two new self-adaptive projection methods for variational inequality problems. Computers and Mathematics with Applications 43, 1529-1537.



\bibitem[Han, 2013]{Handissertation} Han, K., 2013. An analytical approach to sustainable transportation network design. PhD dissertation, Pennsylvania State University.

\bibitem[Han and Friesz, 2012]{LDMcont} Han, K., Friesz, T.L., 2015. Continuity of the effective path delay operator for networks based on the link delay model. Preprint available at http://arxiv.org/abs/1211.4621

\bibitem[Han et al., 2013a]{GVM1} Han, K., Friesz, T.L., Yao, T., 2013a. A partial differential equation formulation of Vickrey's bottleneck model, part I: Methodology and theoretical analysis. Transportation Research Part B 49, 55-74.


\bibitem[Han et al., 2013b]{GVM2} Han, K., Friesz, T.L., Yao, T., 2013b.  A partial differential equation formulation of Vickrey's bottleneck model, part II: Numerical analysis and computation. Transportation Research Part B 49, 75-93.


\bibitem[Han et al., 2013c]{existence} Han, K., Friesz, T.L., Yao, T., 2013c. Existence of simultaneous route and departure choice dynamic user equilibrium. Transportation Research Part B 53, 17-30. 



\bibitem[Han et al., 2014]{spillback} Han, K., Friesz, T.L., Yao, T., 2014. Vehicle spillback on dynamic traffic networks and what it means for dynamic traffic assignment models. 5th International Symposium on Dynamic Traffic Assignment. June 18-20, Salerno, Italy. 



\bibitem[Han et al., 2015]{LKWM} Han, K., Piccoli, B., Szeto, W.Y., 2015.  Continuous-time link based kinematic wave model: Formulation, solution existence and well-posedness. Transportmetrica B, to appear.


\bibitem[Han et al., 2011]{HUD} Han, L., Ukkusuri S., Doan K., 2011. Complementarily formulations for the cell transmission model based dynamic user equilibrium with
departure time choice, elastic demand and user heterogeneity.  Transportation Research Part B 45 (10), 1749-1767.


\bibitem[Huang and Lam, 2002]{HL} Huang, H.J., Lam, W.H.K., 2002. Modeling and solving the dynamic user equilibrium route and departure time choice problem in network with queues. Transportation Research Part B 36 (3), 253-273.




\bibitem[Jang et al., 2005]{JRC} Jang, W., Ran, B., Choi, K., 2005. A discrete time dynamic flow model and a formulation and solution method for dynamic route choice. Transportation Research Part B 39 (7), 593-620.




\bibitem[Konnov, 1998]{Konnov1998} Konnov, I.V., 1998. On quasimonotone variational inequalities. Journal of Optimization Theory and Applications 99 (1), 165-181.


\bibitem[Konnov, 2003]{Konnov} Konnov, I.V., 2003. Application of the proximal point method to non monotone equilibrium problems. Journal of  Optimization Theory and Applications 119, 317-333



\bibitem[Lighthill and Whitham, 1955]{LW} Lighthill, M., Whitham, G., 1955. On kinematic waves. II. A theory of traffic flow on long crowded roads. Proceedings of the Royal Society of London: Series A 229 (1178), 317- 345.



\bibitem[Lo and Szeto, 2002a]{LS2002a} Lo, H.K., Szeto, W.Y., 2002a. A cell-based variational inequality formulation of the dynamic user optimal assignment problem. Transportation Research Part B, 36 (5), 421-443.

\bibitem[Lo and Szeto, 2002b]{LS} Lo, H.K., Szeto, W.Y., 2002b. A cell-based dynamic traffic assignment model: formulation and properties. Mathematical and Computer Modelling, 35(7-8), 849-865.
 
\bibitem[Long et al., 2013]{LHGS} Long, J.C., Huang H.J., Gao, Z.Y., Szeto, W.Y., 2013. An intersection-movement-based dynamic user optimal route choice problem. Operations Research 61 (5), 1134 - 1147.



\bibitem[Mounce, 2006]{Mounce} Mounce, R., 2006. Convergence in a continuous dynamic queuing model for traffic networks. Transportation Research Part  B 40 (9), 779-791.


\bibitem[Mounce, 2007]{Mounce2007} Mounce, R., 2007. Existence of equilibrium in a continuous dynamic queueing model for traffic networks. In: Mathematics in Transport, 219-229, Elsevier, Ed: B.G.Heydecker.


\bibitem[Mounce and Carey, 2011]{MC2011} Mounce, R., Carey, M., 2011. Route swapping in dynamic traffic networks. Transportation Research Part B 45 (1), 102-111.

\bibitem[Mounce and Carey, in press]{MC} Mounce, R., Carey, M., in press. On the convergence of the method of successive averages for calculating equilibrium in traffic networks. Transportation Science. http://dx.doi.org/10.1287/trsc.2014.0517.



\bibitem[Mounce and Smith, 2007]{MS2007} Mounce, R., Smith, M., 2007. Uniqueness of equilibrium in steady state and dynamic traffic networks. In: Transportation and Traffic Theory, 281-299, Elsevier, Eds: R.E.Allsop, M.G.H.Bell and B.G.Heydecker.




\bibitem[Newell, 1993]{Newella}
Newell, G.F., 1993.  A simplified theory of kinematic waves in highway traffic, part I: General theory. Transportation Research Part B 27 (4), 281-287.



\bibitem[Nie and Zhang, 2010]{NZ2010} Nie, Y., Zhang, H.M., 2010. Solving
the dynamic user optimal assignment problem considering queue spillback.
Networks and Spatial Economics 10, 49-71.



\bibitem[Pang and Stewart, 2008]{PS} Pang, J.S., Stewart, D.E., 2008. Differential variational
inequalities. Mathematical Programming, Series A 113 (2), 345-424.









\bibitem[Perakis and Roels, 2006]{PR} Perakis, G., Roels, G., 2006. An analytical model for traffic delays and the dynamic user equilibrium problem. Operations Research 54 (6), 1151-1171.



\bibitem[Pini and Singh, 1997]{PS1997}  Pini, R., Singh, C., 1997. A survey of recent (1985-1995) advances in generalized convexity with applications to duality theory and optimality conditions, Optimization 39, 311-360.


\bibitem[Richards, 1956]{Richards} Richards, P.I., 1956. Shockwaves on the highway. Operations Research 4 (1), 42-51.


\bibitem[Royden and Fitzpatrick, 1988]{Royden} Royden, H.L., Fitzpatrick, P., 1988. Real Analysis (Vol. 3). Englewood Cliffs, NJ:: Prentice Hall. 




\bibitem[Rudin, 2006]{Rudin} Rudin, W., 2006. Functional Analysis. McGraw-Hill. 



\bibitem[Small, 1982]{Small} Small, K.A., 1982. The scheduling of consumer activities: Work trips. American Economic Review 72, 467-479.


\bibitem[Smith and Wisten, 1995]{SW1995} Smith, M.J., Wisten, M.B., 1995. A continuous day-to-day traffic assignment model and the existence of a continuous dynamic user equilibrium, Annals of Operations Research 60 59-79.


\bibitem[Szeto, 2003]{Szeto2003} Szeto, W.Y., 2003. Dynamic Traffic Assignment: Formulations, properties, and extensions. PhD Thesis, The Hong Kong University of Science and Technology, China. 


\bibitem[Szeto et al., 2011]{Szeto2011} Szeto, W.Y., Jiang, Y., Sumalee, A., 2011. A cell-based model for multi-class doubly stochastic dynamic traffic assignment. Computer-Aided Civil and Infrastructure Engineering, 26, 595-611.


\bibitem[Szeto and Lo, 2004]{SL} Szeto, W.Y., Lo, H.K., 2004. A cell-based simultaneous route and departure time choice model with elastic demand. Transportation Research Part B 38, 593-612.


\bibitem[Szeto and Lo, 2006]{SL2006} Szeto, W.Y., Lo, H.K., 2006. Dynamic traffic assignment: Properties and extensions. Transportmetrica 2 (1), 31-52.



\bibitem[Tian et al., 2012]{THG} Tian, L.J., Huang, H.J., Gao, Z.Y., 2012. A cumulative perceived value-based dynamic user equilibrium model considering the travelersÕ risk evaluation on arrival time. Networks and Spatial Economics 12 (4), 589-608.



\bibitem[Ukkusuri et al., 2012]{Ukkusuri} Ukkusuri, S., Han, L. Doan, K., 2012. Dynamic user equilibrium with a path based cell transmission model for general traffic networks.  Transportation Research Part B 46 (10),  1657-1684.


\bibitem[Vickrey, 1969]{Vickrey} Vickrey, W.S., 1969. Congestion theory and transport investment. The American Economic Review 59 (2), 251-261.





\bibitem[Wie et al., 2002]{WTC} Wie, B. W., Tobin R. L., Carey. M., 2002. The existence, uniqueness and computation of an arc-based dynamic network user equilibrium
formulation. Transportation Research Part B 36, 897-918.



\bibitem[Yperman et al., 2005]{LTM} Yperman, I.,  Logghe, S., Immers, L., 2005. The link transmission model: An efficient implementation of the kinematic wave theory in traffic networks. Advanced OR and AI Methods in Transportation, Proc. 10th EWGT Meeting and 16th Mini-EURO Conference, Poznan, Poland, 122-127, Publishing House of Poznan University of Technology.




\bibitem[Yang and Huang, 1997]{YH} Yang, H., Huang, H.J., 1997. Analysis of the time-varying pricing of a bottleneck with elastic demand using optimal control theory.  Transportation Research  Part B 31 (6), 425-440.


\bibitem[Yang and Meng, 1998]{YM} Yang, H.,  Meng Q., 1998. Departure time, route choice and congestion toll in a queuing network with elastic demand. Transportation Research Part B 32 (4), 247-260.


\bibitem[Zhao and Hu, 2007]{ZH} Zhao, Y.B., Hu, J., 2007. Global bounds for the distance to solutions of co-coercive variational inequalities. Operations Research Letters 35 (3), 409-415.


\bibitem[Zhu and Marcotte, 2000]{ZM} Zhu, D.L., Marcotte, P., 2000. On the existence of solutions to the dynamic user equilibrium problem. Transportation Science 34 (4), 402-414.





\end{thebibliography}



\end{document}